\documentclass[11pt]{amsart}
\usepackage{euscript}
\usepackage{amssymb}
\usepackage{amsmath}
\usepackage{epic}
\usepackage{graphics}
\usepackage{epsfig}
\usepackage{color}

\usepackage{amscd,euscript}
\usepackage[frame,cmtip,curve,arrow,matrix,line,graph]{xy}

\usepackage[matrix,arrow,curve]{xy}
\usepackage{mathabx}

\numberwithin{equation}{section}

\newtheoremstyle{my}{1.5em}{0.5em}{\em}{}{\sc}{.}{0.5em}{}

\newtheorem{thm}{Theorem}[section]
\newtheorem{Theorem}[thm]{Theorem}
\newtheorem*{Theorem*}{Theorem}

\newtheorem*{corollary*}{Corollary}

\newtheorem{Proposition}[thm]{Proposition}

\newtheorem*{conjecture*}{Conjecture}
\newtheorem{Question}[thm]{Question}
\newtheorem*{question*}{Question}

\newtheorem*{definitions*}{Definitions}

\newtheorem*{rem*}{Remark}
\newtheorem{Remark}[thm]{Remark}

\newtheorem*{remark*}{Remark}

\newtheorem*{remarks*}{Remarks}
\newtheorem*{example*}{Example}
\newtheorem{Example}[thm]{Example}
\newtheorem*{examples*}{Examples}

\newtheorem*{convention*}{Convention}
\newtheorem*{conventions*}{Conventions}

\newtheorem*{Notes}{Selected references}
\newtheorem*{bibliographical-note*}{Bibliographical note}

\newcommand{\Acknowledgements}{{\em Acknowledgements.} }

\newcommand{\scrB}{\EuScript{B}}

\newcommand{\scrQ}{\EuScript{Q}}
\newcommand{\scrA}{\EuScript{A}}
\newcommand{\scrT}{\EuScript{T}}

\newcommand{\scrX}{\EuScript{X}}

\newcommand{\scrY}{\EuScript{Y}}

\newcommand{\scrL}{\EuScript{L}}

\newcommand{\cF}{\mathcal{F}}

\newcommand{\bR}{\mathbb{R}}
\newcommand{\bZ}{\mathbb{Z}}

\newcommand{\bC}{\mathbb{C}}
\newcommand{\bN}{\mathbb{N}}
\newcommand{\bP}{\mathbb{P}}

\newcommand{\iso}{\cong}           
\newcommand{\smooth}{C^\infty}

\newcommand{\tr}{\mathrm{tr}}
\newcommand{\cdbar}{\mathrm{\overline{\partial}}}

\newcommand{\id}{\mathrm{id}}

\renewcommand{\ker}{\mathrm{ker}}

\newcommand{\Hom}{\mathrm{Hom}}
\newcommand{\End}{\mathrm{End}}

\newcommand{\Ob}{\mathrm{Ob}}

\newcommand{\Symp}{\mathrm{Symp}}
\newcommand{\Diff}{\mathrm{Diff}}

\newcommand{\scrM}{\EuScript{M}}
\newcommand{\scrE}{\EuScript{E}}
\newcommand{\scrF}{\EuScript{F}}
\newcommand{\scrJ}{\EuScript{J}}
\newcommand{\scrG}{\EuScript{G}}

\newcommand{\Cech}[1]{\check{#1}}

\newcommand{\bK}{\mathbb{K}}

\numberwithin{equation}{section}

\renewcommand{\leq}{\leqslant}
\renewcommand{\geq}{\geqslant}

\newcommand{\tensor}{\otimes}

\newcommand{\lra}{\longrightarrow}

\newcommand{\D}{{\mathcal{D}}}

\newcommand{\CC}{\mathcal C}

\renewcommand{\P}{\mathcal{P}}

\newcommand{\Z}{\mathbb{Z}}

\newcommand{\Tw}{\operatorname{Tw}}
\newcommand{\Ext}{\operatorname{Ext}}

\renewcommand{\sc}{\operatorname{sc}}
\newcommand{\Auteq}{\mathrm{Auteq}}

\newcommand{\Stab}{\operatorname{Stab}}

\newcommand{\st}{\operatorname{st}}

\usepackage{epigraph}

\usepackage{etoolbox}

\makeatletter
\patchcmd{\epigraph}{\@epitext{#1}}{\itshape\@epitext{#1}}{}{}
\makeatother

\title{A symplectic prolegomenon}
 \author{Ivan Smith}
 \address{Ivan Smith, Centre for Mathematical Sciences, University of Cambridge, CB3 0WB, England.}
\email{is200@cam.ac.uk}
\date{January 2014 (v1). Revised September 2014 (v2).}
\begin{document}
\maketitle

\epigraph{...everything you wanted to say required a context.  If you gave the full context, people thought you a rambling old fool. If you didn't give the context, people thought you a laconic old fool.}{\emph{Julian Barnes}, Staring at the Sun}


\begin{abstract} A symplectic manifold gives rise to a triangulated $A_{\infty}$-category, the derived Fukaya category, which encodes information on Lagrangian submanifolds and dynamics  as probed by Floer cohomology. This survey aims to give some insight into what the Fukaya category is, where it comes from and what symplectic topologists want to do with it.\end{abstract}

\section{Introduction}

The origins of symplectic topology lie in classical Hamiltonian dynamics, but it also arises naturally in algebraic geometry, in representation theory, in low-dimensional topology and in string theory. In all of these settings, arguably the most important  technology deployed in symplectic topology is Floer cohomology, which associates to a pair of oriented $n$-dimensional Lagrangian submanifolds  $L, L'$ of a $2n$-dimensional symplectic manifold $(X,\omega)$ a $\bZ_2$-graded vector space $HF^*(L,L')$,  which  categorifies the classical homological intersection number of $L$ and $L'$ in the sense that $\chi(HF^*(L,L')) = (-1)^{n(n+1)/2}\, [L]\cdot[L']$. 

In the early days of the subject, Floer cohomology played a role similar to that taken by singular homology groups in algebraic topology at the beginning of the last century: the ranks of the groups provided lower bounds for problems of geometric or dynamical origin (numbers of periodic points, for instance).  But, just as with its classical algebro-topological sibling, the importance of Floer cohomology has increased as it has taken on more structure.  Floer cohomology groups for the pair $(L,L)$ form a unital ring, and $HF^*(L,L')$ is a right module for that ring;  Floer groups for different Lagrangians are often bound together in long exact sequences; the groups are functorial under morphisms of symplectic manifolds defined by Lagrangian correspondences, and are themselves morphism groups in a category whose objects are Lagrangian submanifolds, the Donaldson category, associated to $X$.  

More recently, attention has focussed on more sophisticated algebraic structures,  and on a triangulated $A_{\infty}$-category $\scrF(X)$ underlying the Donaldson category, known as the Fukaya category.  $A_{\infty}$-structures occur elsewhere in mathematics, and -- largely in the wake of Kontsevich's remarkable Homological Mirror Symmetry conjecture -- it is now understood that symplectic topology will be glimpsed, through a glass darkly, by researchers in other fields, precisely because categories of natural interest in those subjects can be interpreted as Fukaya categories of symplectic manifolds (often, of symplectic manifolds which at first sight are nowhere to be seen in the sister theory).  Whilst the formulation of mirror symmetry makes the relevance of Fukaya categories manifest to algebraic geometers working with derived categories of sheaves, there are also connections to geometric representation theorists interested in the Langlands program, to cluster algebraists studying quiver mutation, to low-dimensional topologists working in gauge theory or on quantum invariants of knots, and beyond.  Perversely, for some the Fukaya category might be the first point of contact with symplectic topology at all.

For the motivated graduate student aiming to bring  the Fukaya category into her  daily toolkit, there are both substantive treatments \cite{FO3, FCPLT} and gentler introductions \cite{Auroux} already available.  This more impressionistic survey might nevertheless not be entirely unhelpful.    We hope it \emph{will}    motivate those using holomorphic curve theory to engage with the additional effort involved in adopting a categorical perspective: we try to show how the increase in algebraic complexity extends the reach of the theory, and what it means to ``compute" a category.  More broadly, 
we hope to render the Fukaya category more accessible -- or at least more familiar -- to those, perhaps in adjacent areas,  who wish to learn something of what symplectic topology is concerned with, of how triangulated categories arise in that field and what we might hope to do with them.  

\noindent \textbf{Apologies.}  Floer theory is full of subtleties, which we either   ignore\footnote{\emph{The abbreviators of works do injury to knowledge}, Leonardo da Vinci, c.1510.} or identify and then skate over, in the hope of not obscuring a broader perspective\footnote{I am grateful to the anonymous referee for several improvements to my skating style.}.  This is a rapidly evolving field and the snapshot given here is far from comprehensive.

\vspace{0.2cm}
\Acknowledgements My view of the Fukaya category has been predominantly shaped through conversations and collaborations with Mohammed Abouzaid, Denis Auroux and Paul Seidel.  It has been a  privilege to learn the subject in their wake.  

\section{Symplectic background}

\subsection{First notions} A symplectic manifold $(M^{2n},\omega)$ is an even-dimensional manifold which is obtained by gluing together open sets in $\bR^{2n}$ by diffeomorphisms whose derivatives lie in the subgroup $Sp_{2n}(\bR) = \left\{A \ | \ A^t J A = J \right\} \leq GL_{2n}(\bR)$, where $J = \left(\begin{array}{cc} 0 & -1 \\ 1 & 0 \end{array} \right)$ is a fixed skew-form on Euclidean space. Taking co-ordinates $z_j = x_j + iy_j$, $1\leq j \leq n$ on $\bC^n$, and letting $\omega_{\st} = \sum_j dx_j \wedge dy_j$, these are the diffeomorphisms which preserve the 2-form $\omega_{\st}$, and in fact a manifold is symplectic if and only if it admits a 2-form $\omega$ which is closed, $d\omega = 0$, and non-degenerate, $\omega^n > 0$.  This already excludes various topologies, starting with the sphere $S^4$.

The single most important feature of symplectic geometry is that, via the identification $T^*M \cong TM$ defined by (non-degeneracy of) $\omega$, any smooth $H: M \rightarrow \bR$ defines a vector field $X_H = \iota_{\omega}(dH)$, whose flow $\phi^t_H$ preserves the symplectic structure. Thus, symplectic manifolds have infinite-dimensional locally transitive symmetry groups, and there are no local invariants akin to Riemannian curvature.  The Darboux-Moser theorem strengthens this, asserting that (i) any symplectic $2n$-manifold  is locally isomorphic to $(B^{2n}, \omega_{\st})$, and (ii) deformations of $\omega$ preserving its cohomology class can be lifted to global symplectomorphisms, at least when $M$ is compact. The interesting topological questions in symplectic geometry are thus  global in nature.  As an alternative worldview,  one could say that any symplectic manifold is locally as complicated as Euclidean space, about which a vast array of questions remain open.

Two especially important classes of symplectic manifolds are:
\begin{enumerate}
\item (quasi-)projective varieties $X \subset \bP^N$, equipped with the restriction of the Fubini-Study form $\omega_{\operatorname{FS}}$; in particular, any finite type oriented two-dimensional surface, where the symplectic form is just a way of measuring area;
\item cotangent bundles $T^*Q$, with the canonical form $\omega =d\theta$, $\theta$ the Liouville 1-form, and more generally Stein manifolds, i.e. closed properly embedded holomorphic submanifolds of $\bC^N$, equivalently complex manifolds equipped with an exhausting strictly plurisubharmonic function $h$, with $\omega = -dd^ch$.
\end{enumerate}

Symplectic manifolds are flexible enough to admit cut-and-paste type surgeries \cite{Gompf}; many such in dimension four yield examples which are provably not algebraic, and one dominant theme in symplectic topology is to understand the inclusions $\{$K\"ahler$\} \subset \{$Symplectic$\} \subset \{$Smooth$\}$.  
For open manifolds, there is a good existence theory for symplectic structures, but questions of uniqueness remain subtle:  there are countably many symplectically distinct (finite type) Stein structures on Euclidean space, and recognising the standard structure is algorithmically undecidable \cite{McLean:Computability}.

\subsection{Second notions}
A second essential feature of symplectic geometry is that any symplectic manifold admits a \emph{contractible} space $\scrJ$ of almost complex structures which tame the symplectic form, in the sense that $\omega(v, Jv) > 0$ for any non-zero tangent vector $v$. Following a familiar pattern, discrete (e.g. enumerative) invariants defined with respect to an auxiliary choice drawn from a contractible space reflect intrinsic properties of the underlying symplectic structure. 

A submanifold $L\subset M$ is Lagrangian if it is isotropic, $\omega|_L = 0$, and half-dimensional (maximal dimension for an isotropic submanifold).  In the classes given above, Lagrangian submanifolds include (1) the real locus of an algebraic variety defined over $\bR$, (2) the zero-section and cotangent fibre in $T^*Q$, respectively.  Gromov's fundamental insight was that, for a generic taming almost complex structure, spaces of non-multiply-covered holomorphic curves in $M$ (perhaps with Lagrangian boundary conditions) are finite-dimensional smooth manifolds, which moreover admit geometrically meaningful compactifications by nodal curves.   It is the resulting invariants which count ``pseudoholomorphic" curves (maps from Riemann surfaces which are $J$-complex linear for some taming but not necessarily integrable $J$) that dominate the field.

Discrete invariants of symplectic manifolds and Lagrangian submanifolds include characteristic classes which determine the virtual, i.e. expected, dimensions of moduli spaces of holomorphic curves (when non-empty), via index theory. Any choice of taming $J$ makes $TM$ a complex bundle, which has Chern classes $c_i(M) \in H^{2i}(M;\bZ)$.  Since the Grassmannian  $U(n) / O(n)$ of Lagrangian subspaces of $\bC^n$ has fundamental group $\bZ$, any Lagrangian submanifold has a Maslov class $\mu_L: \pi_2(M,L) \rightarrow \bZ$, which measures the twisting of the Lagrangian tangent planes around the boundary of a disc $D$ relative to a trivialisation of $TM|_D$; under $H^2(M,L;\bZ) \rightarrow H^2(M;\bZ)$, $\mu_L$  maps to $2c_1(M)$.    The moduli space of \emph{parametrised} rational curves $\bP^1 \rightarrow M^{2n}$ representing a homology class $A$ has real dimension $2\langle c_1(M),A\rangle + 2n$, hence dividing by $\bP SL_2(\bC)$ the space of unparametrized curves has dimension $2\langle c_1(M), A\rangle + 2n - 6$. The space of parametrised holomorphic discs on $L$ in a class $u \in \pi_2(M,L)$ has virtual real dimension $n + \mu(u)$, and carries an action of the 3-dimensional real group $\bP SL_2(\bR)$.

\subsection{Contexts} Why study symplectic topology at all?  

\subsubsection{Lie theory} Alongside their renowned classification of finite-dimensional Lie groups, Lie and Cartan also considered infinite-dimensional groups of symmetries of Euclidean space. To avoid redundancies, they considered groups of symmetries acting locally transitively on $\bR^n$, which did not preserve any non-trivial foliation by parallel planes $\bR^k \subset \bR^n$.  There turn out to be only a handful of such pseudogroups (see \cite{SingerSternberg} for a modern account):  the full diffeomorphism group, volume-preserving diffeomorphisms and its conformal analogue, symplectic or  conformally symplectic transformations in even dimensions, contact transformations in odd dimensions, and holomorphic analogues of these groups. Thus, symplectic geometry arises as one of a handful of ``natural" infinite-dimensional geometries (in contrast to the classification in finite dimensions, there are no ``exceptional" infinite-dimensional groups of symmetries of $\bR^n$). 

\subsubsection{Dynamics} Fix a smooth function $H: \bR^{2n} \rightarrow \bR$. Hamilton's equations are
\[
\frac{\partial H}{\partial q} = -\dot{p}; \ \frac{\partial H}{\partial p} = \dot{q}
\]
where $(q,p) \in \bR^n \times \bR^n = T^*\bR^n$.  If $H = p^2/2 + U(q)$ for a potential function $U: \bR^n \rightarrow \bR$, these  
govern the evolution of a particle of mass 1 with position and momentum co-ordinates $(q,p) \in T^*\bR^n$, subject to a force $ -\partial U /\partial q$.  The identity
$
\frac{dH}{dt} =  \frac{\partial H}{\partial q} \frac{dq}{dt} + \frac{\partial H}{\partial p} \frac{dp}{dt} = 0
$
yields conservation of the energy $H$, and classical Poincar\'e recurrence phenomena follow from the fact that the evolution of the system preserves volume, via preservation of the measure $\omega_{\st}^n = \prod_j dp_j dq_j$; but in fact, the evolution preserves the underlying symplectic form $\omega_{\st}$.  Thus, classical Hamiltonian dynamics amounts to following a path in $\Symp(\bR^{2n})$.  

Given two points $x,y \in Q$, and a time-dependent Hamiltonian function $H_t: T^*Q \rightarrow \bR$, the intersections of the Lagrangian submanifolds $\phi_H^1(T_x^*)$ and $T_y^*$ exactly correspond to the trajectories of the system from $x$ to $y$, so from the dynamical viewpoint Lagrangian intersections are natural from the outset.

\subsubsection{Geometry of algebraic families\label{Sec:AlgFamily}} Let $f: X \rightarrow B$ be a holomorphic map of projective varieties, smooth over $B^0 \subset B$. Then $f$ defines a locally trivial fibration over $B^0$.  In most cases, the smooth fibres of $f$ will be distinct as complex manifolds, but Moser's theorem implies that they are isomorphic as symplectic manifolds; the structure group of the fibration reduces to $\Symp(f^{-1}(b))$. Consider for example the quadratic map $\pi: \bC^3 \rightarrow \bC$, $(z_1,z_2,z_3) \mapsto \sum z_j^2$, which is smooth over $\bC^*$.  The fibres away from zero are smooth affine conics, symplectomorphic to $(T^*S^2, d\theta)$.  The monodromy of the fibration is a \emph{Dehn twist} in the Lagrangian vanishing cycle $S^2 \subset T^*S^2$. A basic observation\footnote{$\Diff_{ct}$ resp. $\Symp_{ct}$ denotes the group of compactly supported diffeomorphisms resp.  symplectomorphisms.} uniting insights of Arnol'd, Kronheimer and Seidel is that this monodromy is of order 2 in $\pi_0 \Diff_{ct}(T^*S^2)$, but is of infinite order in $\pi_0 \Symp_{ct}(T^*S^2)$ \cite{Seidel:twist}, see Section \ref{Sec:Appl1}.  Thus, passing from symplectic to smooth monodromy loses information.  This phenomenon is ubiquitous, and symplectic aspects of monodromy seem inevitable in the parts of algebraic geometry that concern moduli.

The fixed points of a symplectomorphism $\phi: X \rightarrow X$ are exactly the intersection points of the Lagrangian submanifolds $\Delta_X, \Gamma_{\phi}$ in the product $(X\times X, \omega \oplus -\omega)$, connecting  questions of symplectic monodromy with Lagrangian intersections.

\subsubsection{Low-dimensional topology}  There are symplectic interpretations of many invariants in low-dimensional topology (the Alexander polynomial, Seiberg-Witten invariants, etc). Donaldson conjectured \cite[p.437]{McD-S} that two homeomorphic smooth symplectic 4-manifolds $X,Y$ are diffeomorphic if and only if the product symplectic structures on $X\times S^2$ and $Y\times S^2$ are deformation equivalent (belong to the same component of the space of symplectic structures), suggesting that symplectic topology is fundamentally bound up with exotic features of 4-manifold topology.

\subsection{Some open questions}
Many basic questions in symplectic topology seem, to date, completely inaccessible.  For instance, it is unknown whether $\Symp_{ct}(\bR^{2n}, \omega_{\st})$ is contractible, connected or even has countably many components, for $n>2$.  There is no known example of a simply-connected closed manifold of dimension at least 6  which admits almost complex structures and a degree 2 cohomology class of non-trivial top power but no symplectic structure.   Without denying the centrality of such problems, in the sequel we concentrate on questions on which at least some progress has been made.

\subsubsection{Topology of Lagrangian submanifolds} Questions here tend to fall into two types: what can we say about  (the uniqueness, displaceability of...) Lagrangian submanifolds that we see in front of us, and what can we say about (the existence, Maslov class of...) those we don't?  

\begin{Question} What are the Lagrangian submanifolds of $(\bC^n, \omega_{\st})$ or of  $(\bP^n, \omega_{\operatorname{FS}})$ ? \end{Question}

To be even more concrete, we know the only prime 3-manifolds which admit Lagrangian embeddings in $(\bC^3,\omega_{\st})$ are products $S^1\times \Sigma_g$ \cite{Fukaya:LagC3}; for any orientable 3-manifold $Y$, $Y\#(S^1\times S^2)$ admits a Lagrangian embedding in $(\bC^3,\omega_{\st})$ \cite{YETI}; but essentially nothing is known about which connect sums of hyperbolic 3-manifolds admit such Lagrangian embeddings.

\begin{Question}\label{Q:Arnold} If $T^*L \cong_{\omega} T^*L'$ are symplectomorphic, must $L$ and $L'$ be diffeomorphic ?\end{Question}

Arnold's ``nearby Lagrangian submanifold" conjecture would predict ``yes". Abouzaid and Kragh \cite{Kragh,Abouzaid:Maslov0} have proved that $L$ and $L'$ must be homotopy equivalent, and in a few cases --  $L=S^{4n+1}$ or  $L= (S^1\times S^{8n-1})$, see \cite{Abouzaid:exotic,EkholmSmith} --  it is known that $T^*L$ remembers aspects of the smooth structure on $L$, i.e. $T^*L \not\cong_{\omega} T^*(L\#\Sigma)$ for certain homotopy spheres $\Sigma$.

\begin{Question} Can a closed $(X,\omega)$ with $\dim_{\bR}(X)>2$ contain infinitely many pairwise disjoint Lagrangian spheres?
\end{Question}

For complete intersections in projective space there are bounds on numbers of \emph{nodes} (hence simultaneous vanishing cycles) coming from Hodge theory, but the relationship between Lagrangian spheres and vanishing cycles is slightly mysterious, see e.g. \cite{CortiSmith}. 

\subsubsection{Mapping class groups and dynamics} The symplectic mapping class group $\pi_0 \Symp(X)$  is a rather subtle invariant  in that it can change drastically under perturbations of $\omega$ which change its cohomology class (so are not governed by Moser's theorem).

\begin{Question}\label{Question:Monodromy} If $X_d \subset \bP^N$ is the degree $d$ hypersurface, $d>2$, and $\scrX_d$ the coarse moduli space of all such, describe the (co)kernel of the natural  representation $\pi_1(\scrX_d) \rightarrow \pi_0\Symp(X_d)$. 
\end{Question}

It is known the map has infinite image;  it seems unlikely that it is injective in this generality.  The cokernel is even more mysterious: it is unknown whether $\pi_0\Symp(X_d)$ is finitely generated.  Closer to dynamics, one can look for analogues of Rivin's result \cite{Rivin} that a random walk on the classical mapping class group almost surely yields a pseudo-Anosov.  

\begin{Question} \label{Question:RandomWalk}
For a random walk on $\pi_0\Symp(X_d)$, $d>2$, does the number of periodic points of any representative of the final mapping class grow exponentially with probability one?
\end{Question}

A quite distinct circle of dynamical questions arises from questions of \emph{flux}.  An element $\phi \in \Symp_0(M)$, i.e. a symplectomorphism isotopic to the identity, is Hamiltonian if it is the time one map of a Hamiltonian flow $\phi_H^t$.  Fix $a\in H^1(M;\bR)$ and a path $a_t$ of closed 1-forms with $\int_0^1 a_t dt = a$. The $a_t$ are dual to symplectic vector fields $X_t$ which generate an isotopy $\phi^t$, and $\phi^1 = \phi \in \Symp_0(M)$ depends up to Hamiltonian isotopy only on $a$ and not on the path $(a_t)$.  The \emph{flux group} $\Gamma \subset H^1(M;\bR)$ comprise those $a$ for which $\phi^1 \simeq \id$ is Hamiltonian isotopic to the identity; a deep theorem of Ono \cite{Ono} implies $\Gamma$ is discrete.

\begin{Question}
Which  $a\in H^1(M,\bR)$ are realised by a fixed-point-free symplectomorphism?
\end{Question}

Elements of $\Gamma$ cannot be, by the (solution of the) Arnol'd conjecture \cite{FOno, LiuTian}.  Note that flux is probing $\Symp_0(M)$, rather than $\pi_0 \Symp(M)$, in contrast to  Question \ref{Question:RandomWalk}.

\subsection{The phenomenon or philosophy of mirror symmetry}\label{Sec:Mirror}

In its original formulation, Kontsevich's homological mirror symmetry conjecture  \cite{Kontsevich} asserts
\begin{equation} \label{Eqn:HMS}
D^{\pi}\scrF(X) \simeq D^b(\Cech{X})
\end{equation}
where $X, \Cech{X}$ area a pair of mirror Calabi-Yau manifolds,  $D^b(\Cech{X})$ denotes the bounded derived category of coherent sheaves, whilst $D^{\pi}\scrF(X)$ is the split-closed derived Fukaya category. The former has objects finite-length complexes of algebraic vector bundles  $\mathcal{E}^{\bullet} \rightarrow \Cech{X}$, with morphisms being Cech complexes $\Cech{C}^*(\EuScript{U}; \mathcal{H}om(\mathcal{E}^{\bullet}, \mathcal{F}^{\bullet}))$, with $\EuScript{U}$ a fixed affine open cover of $\Cech{X}$; the latter is obtained by a formal algebraic process of enlargement from a category whose objects are  ``unobstructed" Lagrangian submanifolds of $X$ and whose morphisms are the Floer cohomology groups introduced below, see Sections \ref{Sec:Floer} and \ref{Sec:Fukaya}.  The formulation hides the fact that in general given $X$ there is no obvious way of producing $\Cech{X}$, and indeed it may not be unique.  
(Better formulations, due to Gross and Siebert \cite{GrossSiebert},  involve an  involutive symmetry on a certain class of toric degenerations of a Calabi-Yau. The central fibres of those degenerations should be unions of toric varieties which are  dual in a direct combinatorial manner.)  

Geometrically, following Strominger-Yau-Zaslow \cite{SYZ}, the conjectured  relation between $X$ and $\Cech{X}$ is that they carry dual Lagrangian torus fibrations (with singular fibres); the base of such a Lagrangian fibration is an integral affine manifold with singularities, and the toric duality arises from a ``Legendre transform"  on such manifolds.   From this viewpoint, mirror symmetry has been extended to a much wider range of contexts: given a variety $X$ with effective anticanonical divisor $D\subset X$, one looks for a Lagrangian fibration on $U=X\backslash D$ and takes its dual $\Cech{U}$.  Compactification from $U$ back to $X$ is mirror to turning on a holomorphic ``potential" function $W: \Cech{U}\rightarrow \bC$, built out of counts of holomorphic discs with boundary on the Lagrangian torus fibres, and vice-versa. (In the original Calabi-Yau setting, $D = \emptyset$ and $W$ is constant.) There are still categories of Lagrangians and sheaves  which can be compared on the two sides of the mirror: for instance, if $X$ is Fano, one  asks if $D^{\pi}\scrF(X) \simeq D^b_{sing}(W:\Cech{U} \rightarrow \bC)$, the RHS being the homotopy category of matrix factorisations of $W$ \cite{Orlov:Dbsing}.   

Elucidating this would be too much of a digression:  for motivational purposes, it suffices to know that there are algebro-geometric routes to computing categories which are conjecturally, and sometimes provably, equivalent to  categories built out of compact Lagrangian submanifolds, even though one has no realistic hope of classifying or even enumerating such Lagrangians.  Conversely, categories of interest in algebraic geometry can be studied via symplectic methods: symplectic automorphisms of $X$ act on $D^{\pi}\scrF(X)\simeq D^b(\Cech{X})$, and the symplectic mapping class group of $X$ tends to be much richer than the group of algebraic automorphisms of $\Cech{X}$.

\begin{Notes} For symplectic topology basics see \cite{ACdaS, McD-S}; for dynamics and topology see \cite{HZ, EOSS}; for a mirror symmetry overview see \cite{AspinwallEtAl, Ballard, Gross:tropical}. \end{Notes}

\section{Floer cohomology}\label{Sec:Floer}

\subsection{Outline of the definition}

Fix a coefficient field $\bK$. Under suitable technical assumptions,  Floer cohomology associates to a pair of closed oriented Lagrangian submanifolds $L, L' \subset X$  a $\bZ_2$-graded vector space $HF^*(L,L')$ over $\bK$ which has the following basic features:
\begin{enumerate}
\item The Euler characteristic $\chi(HF^*(L,L')) = (-1)^{n(n+1)/2}\, [L]\cdot[L']$ recovers the algebraic intersection number of the homology classes defined by $L, L'$;
\item There is a Poincar\'e duality isomorphism $HF^*(L,L') \cong HF^{n-*}(L',L)^{\vee}$;
\item The space $HF^*(L,L')$ is invariant under Hamiltonian isotopies of either $L$ or $L'$, i.e. $HF^*(L,L') \cong HF^*(\phi(L), L')$ for any Hamiltonian symplectomorphism $\phi$;
\item $HF^*(L, L') = 0$ if $L$ and $L'$ are disjoint; hence its non-vanishing provides an obstruction to Hamiltonian displaceability.
\end{enumerate}

Let $L, L'$ be closed orientable Lagrangian submanifolds of a symplectic manifold $(X,\omega)$. We assume that $X$ is either closed, or convex at infinity in the sense that there is a taming almost complex structure $I$ on $X$, and an exhaustion (a sequence of relatively compact open subsets $U_1 \subset U_2 \subset \cdots$ whose union is $X$), such that if $S$ is a connected compact Riemann surface with nonempty boundary, and $u: S \rightarrow X$ an $I$-holomorphic map with $u(\partial S) \subset U_k$, then necessarily $u(S) \subset U_{k+1}$.  Symplectic structures obtained from plurisubharmonic exhaustive functions on Stein manifolds satisfy these properties \cite{CieliebakEliashberg}. 

Choose a compactly supported time-dependent Hamiltonian $H \in \smooth_{ct}([0,1] \times X,\bR)$, whose associated Hamiltonian isotopy $(\phi^t)$ has the property that $\phi^1(L)$ intersects $L'$ transversally. If char$(\bK) = 2$, the Floer cochain complex is by definition freely generated over $\bK$ by that set of intersection points.  If $\textrm{char}(\bK)\neq 2$, index theory associates to an intersection point (i.e. Hamiltonian chord) $x$ a 1-dimensional $\bK$-vector space $\mathrm{or}_x$ generated by two elements (the ``coherent orientations" at $x$) subject to the relation that their sum vanishes.  For the differential, we choose a generic family $J=\{J_t\}_{t\in[0,1]}$ of taming almost complex structures on $X$. Given two intersection points $x_\pm$, we denote by $\scrM(x_-,x_+)$ the moduli space of Floer trajectories
\begin{equation} \label{eq:Floer}
\begin{aligned}
 & u: \bR \times [0,1] \longrightarrow X, \\
 & u(s,0) \in L, \;\; u(s,1) \in L', \\
 & \partial_s u + J(t,u) \big(\partial_tu - X(t,u)\big) = 0, \\
 & \textstyle\lim_{s \rightarrow \pm\infty} u(s,t) = y_\pm(t), \\
 & E(u) = \int u^*\omega = \int \!\!\! \int \big| \partial u / \partial s \big|^2 \, ds\,dt \ < \infty,
\end{aligned}
\end{equation}
where $X$ denotes the Hamiltonian vector field of $H$, meaning that $\omega(X,\cdot) = dH$, and $y_\pm$ are flow lines of $X$ with $y_\pm(1) = x_\pm$.  Then
\begin{equation} \label{Eqn:FloerDifferential}
d_J(x_+) = \sum_{x_-} n_{x_+, x_-} x_-
\end{equation} where $n_{x_+, x_-} \in \bK$ counts the isolated solutions $u \in \scrM(x_-, x_+) / \bR$ (where the $\bR$-action is by translation in the $s$-variable) with a sign canonically encoded as an isomorphism $\delta_u: \mathrm{or}_{x_+} \stackrel{\cong}{\longrightarrow} \mathrm{or}_{x_-}$. One often counts the solutions weighted by the symplectic area of the corresponding holomorphic disc; for that and the question of convergence of \eqref{Eqn:FloerDifferential} see Section \ref{Sec:Novikov}. 

\begin{center}
\begin{figure}[ht]
\includegraphics[scale=0.3]{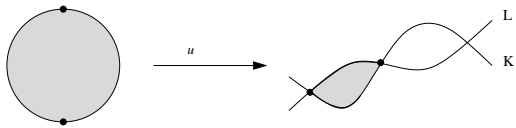}
\caption{The Floer differential in $CF^*(L,K)$\label{Fig:FloerWhitney}}
\end{figure}
\end{center}

Suppose $L \pitchfork L'$ are transverse, and take $H \equiv 0$. Then the equation \eqref{eq:Floer} is just the usual Cauchy-Riemann equation $\cdbar_J(u) = 0$.  There is a conformal identification $\bR\times[0,1] \cong D^2 \backslash \{\pm i\}$, and a solution to the Floer equation extends to a continuous map of the closed disk, taking the points $\pm i$ to the chosen intersection points $x_{\pm} \in L \cap L'$, so in this sense the Floer differential counts holomorphic disks, cf. Figure \ref{Fig:FloerWhitney}. Crucially, (i) for generic data one expects moduli spaces of solutions to \eqref{eq:Floer} to be finite-dimensional smooth manifolds; and (ii) those moduli spaces admit geometrically meaningful compactifications. The smoothness and finite-dimensionality of moduli spaces is a consequence of the ellipticity of the Cauchy-Riemann equations (with totally real boundary conditions), whilst the geometric description of compactifications of spaces of curves with uniformly bounded area is a variant on ideas going back to Deligne and Mumford's introduction of the moduli space of stable curves: the boundary strata index maps from nodal Riemann surfaces (trees of discs and spheres).

It is a non-trivial theorem, under additional hypotheses discussed below in Section \ref{Sec:GeometricHyp}, that $d_J^2 = 0$ so this indeed defines a chain complex.  The mechanism underlying $d_J^2=0$, as in Morse theory, is that one-dimensional moduli spaces of solutions to \eqref{eq:Floer}  have compactifications whose boundary strata label contributions to $d_J^2$; the additional hypotheses rule out other, unwanted, degenerations which would disrupt that labelling.  Given that, the associated cohomology $HF^*(L,L')$ is denoted Floer cohomology. A ``homotopy of choices" argument implies that it is independent both of the choice of $J$ and of Hamiltonian isotopies of either of $L$ or $L'$. Indeed, given either a family $\{J_s\}$ of almost complex structures or a moving Lagrangian boundary condition $\{L_s\}$ constant for $|s| \gg 0$, one considers the continuation map equation
\begin{equation} \label{Eqn:Continuation}
 \partial_s u + J_{s}(t,u) \big(\partial_tu - X_{s}(t,u)\big) = 0
\end{equation}  
with the boundary conditions $\{L_s, L'\}$, and the count of solutions gives a chain map between the relevant Floer complexes. One sees this induces an isomorphism on cohomology, by deforming the concatenation of the map with that induced by the inverse isotopy back to the constant family, which induces the identity on cohomology.

\begin{Example} Let $X=T^*S^1$ and consider the Lagrangians $L, L'$ given by the 0-section and a homotopic circle meeting that transversely twice at points $a,b$.  There are two lunes $u,v$ spanning the circles, which by Riemann mapping are conformally the images of holomorphic strips.  The complex $CF(L, L') = \bK\langle a\rangle \oplus \bK\langle b\rangle$.  Take $\bK = \bC$. Counting holomorphic strips $u$ weighted by their area $e^{-\int_u \omega}$, and taking into account  orientations of  holomorphic strips, one sees that
\[
\partial(a) = (e^{-\omega(u)} - e^{-\omega(v)}) \, b, \qquad \partial(b) = 0.
\]
Therefore $HF(L,L') =0$ unless $\omega(u) = \omega(v)$; indeed, if the areas are different, one can Hamiltonian isotope $L'$ to be a parallel translate to $L$, disjoint from it.  If the areas of the lunes are equal, $L'$ is a Hamiltonian image of $L$, and $HF(L,L')$ has rank 2.
\end{Example}

 A typical paper in Floer theory begins by laying down a number of technical hypotheses on the ambient geometry.  In Sections \ref{Sec:GeometricHyp} and \ref{Sec:AlgebraicHyp} we gather together some of the most common of these,  indicating where they play a role, and what one can achieve in their absence or presence.

\subsection{Motivations for the definition} The first motivation below is historically accurate, but the others provide useful context.

\subsubsection{Morse theory on the path space}  If $Q$ is a smooth finite-dimensional manifold, and $f: Q \rightarrow \bR$ is a generic function which, in particular, has non-degenerate critical points, then the Morse complex $C^*(f)$ of $f$ is generated by the critical points of $f$, and the differential counts gradient flow lines 
\[
\gamma: \bR \rightarrow Q,  \quad \dot{\gamma}(t) = \nabla f_{\gamma(t)}
\]
where the gradient flow is defined with respect to an auxiliary choice of Riemannian metric on $Q$.  There is a generalisation, due to Novikov, in which the function $f$ is replaced by a one-form which is not necessarily exact. The Floer complex is obtained from a formal analogue of that construction, replacing $Q$ by the space of paths $\mathcal{P}(L,L')$ and the function $f$ by the action functional (one-form) $\scrA$ which integrates the symplectic form over an infinitesimal variation of some path $\gamma$, $\xi \mapsto \int_0^1\omega(\xi, \dot{\gamma}) dt$, for $\xi$ a vector field along $\gamma$.  The critical points of $\scrA$ are constant paths, i.e. intersection points. Fix a compatible almost complex structure $J$ on $(X,\omega)$, i.e. $J$ is taming and $\omega(J\cdot, J\cdot) = \omega(\cdot,\cdot)$. Then $J$ defines an $L^2$-metric on the path space, and the formal gradient flow equation is the Cauchy-Riemann equation  \eqref{eq:Floer}, with $H=0$. 

\begin{Remark} \label{Rem:Filtration}
The gradient flow viewpoint on $HF^*(L,L')$ brings out an additional structure: if $\scrA$ is single-valued the complex $CF^*(L,L')$ is naturally filtered by action, i.e. by the values of the functional $\scrA$, and the Floer differential respects that filtration.
\end{Remark}

\subsubsection{Geometric intersection numbers} Let $\Sigma_g$ be an oriented surface of genus $g$. For (homotopically non-trivial) curves $\alpha, \beta$ on $\Sigma$, let $\iota(\alpha,\beta) \in \bZ_{\geq 0}$ denote their geometric intersection number.  A diffeomorphism $f$ defines a sequence $\iota(\alpha, f^n(\beta))$ for $n\in \bN$, which depends only on the mapping class $[f] \in \pi_0\Diff(\Sigma)$ of $f$.  The Nielsen-Thurston classification can be formulated dynamically as follows \cite{FLP}: each $f$ satisfies one of the following.

\begin{itemize}
\item $f$ is reducible if some power of $f$ preserves a simple closed curve up to isotopy;
\item $f$ is periodic if for every $\alpha, \beta$ the sequence $\{\iota(\alpha, f^n\beta)\}$ is periodic in $n$;
\item $f$ is pseudo-Anosov if for every $\alpha,\beta$ the sequence $\{\iota(\alpha, f^n\beta)\}$ grows exponentially in $n$.
\end{itemize}

The classification should be viewed as inductive in the complexity of the surface, with the reducible case indicating that  (a power of) $f$ is induced by a diffeomorphism of a simpler surface.  In the most interesting, pseudo-Anosov case, the logarithm of the growth rate of the geometric intersection numbers is independent of $\alpha,\beta$, and is realised by the translation length in the Teichm\"uller metric of the action of the mapping class $f$ on Teichm\"uller space.   Any diffeomorphism of $\Sigma$ is isotopic to an area-preserving diffeomorphism. For curves $\alpha, \beta$ which are not homotopic (or homotopically trivial), $\iota(\alpha,\beta) = \mathrm{rk}\, HF(\alpha,\beta)$, so the Thurston classification is also a classification of symplectic mapping classes in terms of the growth rate of Floer cohomology.  It is a long-standing challenge to find analogues in higher dimensions.

\subsubsection{Surgery theory}   The classification of high-dimensional manifolds relies on the Whitney trick. Consider two orientable submanifolds $K^k, N^{d-k} \subset M^d$ of complementary dimension in a simply-connected oriented manifold $M$, with $d>4$ and $k>2$. Suppose $K,N$ meet in a pair of points $\{x,y\}$ of opposite intersection sign; given arcs between $x$ and $y$ in $K$ and $N$, the Whitney trick  slides $K$ across a disc spanning the arcs (the disc exists by simple-connectivity of $M$, and can be taken embedded and disjoint from $K,N$ in its interior since $d > 4$ and $k>2$), to reduce the geometric intersection number of the submanifolds.  Floer cohomology imagines a refined version of the Whitney trick, in which one would only be allowed to cancel excess intersections between Lagrangian submanifolds by sliding across \emph{holomorphic} Whitney disks. 

The surgery-theoretic aspects of Floer cohomology itself have yet to be developed,  i.e. given a pair of Lagrangian submanifolds $L, L'$ which meet transversely at exactly two points and bound a unique holomorphic disk, it is unknown when the intersection points can be cancelled by moving $L$ by Hamiltonian isotopy.  This is one of the most basic gaps in the theory.

\subsubsection{Quantum cohomology}  Floer cohomology is an open-string (concerning genus zero Riemann surfaces with boundary) analogue of a closed-string invariant (defined through counts of rational curves).  The quantum cohomology of a closed symplectic manifold is  additively  $H^*(X;\bK)$ (with $\bK$ typically a Novikov field, cf. \eqref{Eqn:NovikovField}), but is equipped with a product $\ast$ which deforms the classical intersection product by higher-order contributions determined by 3-point Gromov-Witten invariants. For smooth projective varieties it can be defined purely algebro-geometrically, via intersection theory on moduli stacks of stable maps.   If $\phi: X\rightarrow X$ is a symplectomorphism, then both the diagonal $\Delta_X$ and the graph $\Gamma_{\phi}$ of $\phi$ are Lagrangian in $(X,\omega) \times (X,-\omega)$, and we define\footnote{There is a direct construction of $HF^*(\phi)$ which is more elementary and more generally applicable, but for brevity we will content ourselves with the description via Lagrangian Floer cohomology.} $CF^*(\phi) = CF^*(\Delta_X, \Gamma_{\phi})$.  The generators of $CF^*(\phi)$ are given by the fixed points of $\phi$, assuming these are non-degenerate. To connect the two discussions, there is an isomorphism $HF^*(\id) \cong QH^*(X)$. 

Quantum cohomology is an important invariant in its own right:  it distinguishes interesting collections of symplectic manifolds (e.g. showing the moduli space of symplectic structures modulo diffeomorphisms has infinitely many components on  $K \times S^2$, with $K \subset \bC\bP^3$ a quartic $K3$ surface); it is invariant under 3-fold flops; and it has deep connections to integrable systems.   Viewed as a far-reaching generalisation of $QH^*(X) \cong HF^*(\Delta_X, \Delta_X)$, Floer cohomology for general pairs of Lagrangian submanifolds has the disadvantage of rarely being amenable to algebro-geometric computation, at least prior to insights from mirror symmetry.

\subsection{Technicalities with the definition: Geometric hypotheses\label{Sec:GeometricHyp}}

Many invariants in low-dimensional topology and geometry are defined by counting solutions to an elliptic partial differential equation. Such a definition presupposes that, in good situations, the moduli spaces are compact, zero-dimensional manifolds, which it makes sense to count.  The heart of the matter is thus to overcome transversality (to make the solution spaces manifolds of known dimension) and compactness (to make the zero-dimensional solution spaces finite sets).

\subsubsection{Transversality.} A taming almost complex structure $J$ on $X$ is \emph{regular} for a particular moduli problem (e.g. of maps $u: \bP^1 \rightarrow X$ or of discs $u: (D^2,\partial D^2) \rightarrow (X,L)$ representing a fixed relative homology class which solve $\cdbar_J(u) = 0$) if the linearisation $D_u = D(\cdbar_J)_u$ of the Cauchy-Riemann or Floer equation  is surjective at every solution $u$.  In that setting, the Sard-Smale implicit function theorem shows that the moduli space of solutions $u$ is locally a smooth manifold, of dimension given by the index of the operator $D_u$.  The Sard-Smale theorem pertains to maps between Banach manifolds, so there is a routine complication: one should first extend the $\cdbar$-operator to a section of a Banach bundle $\cdbar_J: L^{k,p}(\Sigma, X) \rightarrow L^{k-1,p}(\mathcal{E})$, with $\mathcal{E}_u = \Omega^{0,1}(u^*TX)$, on a space of $L^{k,p}$ maps (if $kp>2$ the maps are continuous, so pointwise boundary conditions make sense), thereby obtain for regular $J$ a smooth moduli space of solutions which \emph{a priori} is composed of maps of low regularity, and then use ``bootstrapping" to infer that any $L^{k,p}$-solution of the equation is actually a smooth map.

There are  intrinsic obstructions to achieving ``as much transversality as you'd like".  Suppose for instance $X$ is a closed symplectic 6-manifold with $c_1(X) =0$. For any $A \in H_2(X;\bZ)$, the virtual dimension for unparametrised rational curves in class $A$ is $0$. However, if  $u: \bP^1 \rightarrow X$ represents a class $A$, then the class $kA$ is represented by any map $\bP^1 \stackrel{k:1}{\longrightarrow} \bP^1 \stackrel{u}{\longrightarrow} X$, so the moduli space of curves in class $kA$ is at least as large as the Hurwitz space of degree $k$ covers of $\bP^1$ over itself, and cannot have dimension zero.  This is symptomatic of the fact that one cannot hope to achieve transversality at multiply covered curves without perturbing the Floer or Cauchy-Riemann equation itself in a way that breaks the symmetry of the cover (for instance, one can make the almost complex structure $J$ depend on a point of the domain of the curve, so it will not be invariant under any finite symmetry).  For rational curves, a generic choice of $J$ will be regular for simple (non-multiple-cover) curves.  For discs, viewed as maps of the strip $D^2 \backslash \{\pm 1\}$ which extend smoothly over the boundary punctures, a generic path $\{J_t\}_{t\in[0,1]}$ of complex structures will again be regular for solutions which have an injective point $x$, i.e. one where $u^{-1}(u(x)) = \{x\}$. 

\begin{Example} \label{Ex:TransversalityFiniteGroup} There are also transversality issues not related to multiple covers. Suppose $X$ carries a symplectic involution $\iota$ with fixed locus $X^{\iota}$, and one has $\iota$-invariant Lagrangians $L_i$ with $L_i^{\iota} = L_i \cap X^{\iota}$ Lagrangian in $X^{\iota}$.  If for $x,y \in L_1^{\iota} \cap L_2^{\iota}$ the virtual dimension of the moduli space containing a Floer trajectory $u$ inside $X^{\iota}$  is greater than the corresponding virtual dimension of $u$ in $X$, then one cannot simultaneously have that the curve is regular in both spaces: the dense locus of regular $J$ may be disjoint from the infinite-codimension locus of $\iota$-invariant $J$. (If there are no trajectories in the fixed locus, one \emph{can} achieve equivariant transversality \cite{Khovanov-Seidel}.)
\end{Example}

\subsubsection{Compactness via bubbling.} Moduli spaces of solutions, even when regular, will typically not be compact.  Gromov compactness \cite{Gromov} produces a compactification which includes nodal curves (trees of spheres and discs).  An important point is that although the Floer equation is a deformed version of the Cauchy-Riemann equation, the bubbles arise from a rescaling process which means that they satisfy the undeformed equation $\overline{\partial}_J(u) = 0$.  In consequence, one cannot typically hope to overcome transversality problems with multiply covered bubbles by generic choices of inhomogeneous terms or domain-dependent choices of almost complex structures.  Typically, boundary strata for compactified spaces of holomorphic curves may have dimension far in excess of that of the open stratum one is trying to compactify. There are no easy general solutions to this issue (the solutions that exist \cite{FO3:technical, HWZ}, which involve perturbing moduli spaces of solutions inside ambient Banach spaces of maps rather than perturbing the equations or other geometric data themselves, go by the general name of \emph{virtual perturbation theory}).

Concretely, to define Floer cohomology, one needs to know that $d^2 = 0$ in the Floer complex.  As in Morse theory, the proof of this relation comes from studying boundaries of one-dimensional  (modulo translation) moduli spaces of solutions.  Along with the boundary strata which contribute to $d^2$, there is a problematic kind of breaking, in which a holomorphic disc with boundary on either $L$ or $L'$ bubbles off, cf. Figure \ref{Fig:Bubble}.  (The degenerate case in which the boundary values of the disc are mapped by a constant map to the Lagrangian, giving a rational curve inside $X$ which passes through $L$, is also problematic.)  
\begin{center}
\begin{figure}[ht]
\includegraphics[scale=0.4]{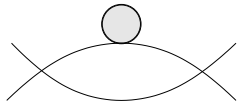}
\caption{Holomorphic disc bubbles can obstruct $d^2 = 0$ in the Floer complex.\label{Fig:Bubble}}
\end{figure}
\end{center}

There are three basic routes to avoiding the problems caused by disc bubbling.

\subsubsection{ Preclude bubbles a priori: asphericity or exactness.\label{Sec:exact}}  Here one insists $\langle \omega, \pi_2(X,L)\rangle = 0$ for all Lagrangians under consideration (which in particular implies that there are no symplectic 2-spheres in $X$).   Since bubbles are holomorphic, if non-constant they must have strictly positive area, by the basic identity 
\[
\int_C u^*\omega \ = \ \int_C \| du\|_J^2
\] 
for a $J$-holomorphic curve $u: C \rightarrow X$ with Lagrangian boundary and $J$ compatible with $\omega$. This was the setting adopted in the earliest papers in the subject, in particular by Floer himself.

For a non-compact symplectic manifold $X$, if the symplectic form $\omega = d\theta$ is globally exact and the Lagrangian $L\subset X$ is exact meaning $\theta|_L = df$ vanishes in $H^1(L)$, bubbles are immediately precluded by Stokes' theorem.  (However, $\omega^n > 0$ shows that exactness of $\omega$ is incompatible with $X$ being a closed manifold.)  Non-compactness of $X$ introduces another possible source of loss of compactness for moduli spaces, which is just that solutions might escape to infinity; this is controlled by a different mechanism, typically some version of the maximum principle, which requires imposing control on the geometry at infinity. Stein manifolds $X$ (in particular affine varieties) are modelled on a contact cone at infinity, i.e. outside a compact set they are symplectomorphic to a subset of $(Y\times \bR_+, d(e^t\alpha))$ with $(Y,\alpha)$ contact and $t\in\bR_+$.  For a holomorphic curve $u: \Sigma \rightarrow X$, the (partially defined) function $t\circ u$ is subharmonic, $\Delta(t\circ u) \geq 0$, which means the function has no interior maxima. It follows that  holomorphic curves whose boundary is contained in the interior $X \backslash (X \cap Y\times\bR_+)$ are entirely contained in the interior, so Stein manifolds are convex at infinity in the sense introduced earlier.

An argument going back to Floer  identifies holomorphic strips for a pair $L, \phi_H^t(L)$ (with $\phi_H^t$ the small-time flow of a Morse function $H$ on $L$, and for a carefully chosen time-dependent almost complex structure $J$) which stay close to $L$ with gradient flow lines of $H$. We may therefore identify $CF^*(L,L) = C^*_{Morse}(L)$ as vector spaces.  In the absence of holomorphic discs with boundary on $L$, a compactness argument shows that there are no other holomorphic strips which can contribute to the Floer differential, and hence:

\begin{Theorem}[Floer] \label{Lem:Floer}
If $\langle \omega, \pi_2(X,L) \rangle = 0$,  then $HF^*(L,L) \cong H^*(L)$. 
\end{Theorem}

This recovers Gromov's theorem \cite{Gromov} that $\bC^n$ contains no closed exact Lagrangian:  it would be displaceable, with $HF(L,L)=0$.   (Any isotopy of exact Lagrangians $\{L_t\} \subset X$ is induced by a Hamiltonian flow of $X$, in particular exact isotopies do not change Floer cohomology.)

\subsubsection{Preclude bubbles for reasons of dimension: monotonicity.\label{Sec:Monotone}}  If the possible disc and sphere bubbles sweep out subsets of $X$ of sufficiently high (co)dimension, then for \emph{generic} choices of data in the Floer equation (almost complex structure, Hamiltonian terms), the boundary values of holomorphic strips that appear in one-dimensional moduli spaces will never interact with bubbles.  The typical setting for this resolution is the case when $X$ and $L$ are \emph{monotone}, meaning that the Maslov index $\mu_L$ and area homomorphisms are positively proportional on $\pi_2(X,L)$, in particular $\omega_X$ and $c_1(X)$ are positively proportional on $\pi_2(X)$, which is a symplectic analogue of the Fano condition in algebraic geometry.  Hypersurfaces or complete intersections in $\bP^n$ of sufficiently low total degree $\leq n$ are monotone.  The minimal Chern number $N_{min}(X)>0$ of a monotone symplectic manifold is defined by the equality $\{\langle c_1(X), \pi_2(X)\rangle\} = N_{min}(X)\bZ$.  Similarly, a monotone Lagrangian $L\subset X$ has a minimal Maslov number $N_{min}(L)>0$.  Since $\mu_L \mapsto 2c_1(X)$ under $H^2(X,L)\rightarrow H^2(X)$,  $N_{min}(L) = 2N_{min}(X)$ when $\pi_1(L)=0$.

Suppose first $X$ is monotone with $N_{min}(X) \geq 2$, and $K, L\subset X$ are monotone Lagrangians each having $N_{min}\geq 4$. The locus of configurations depicted in Figure \ref{Fig:Bubble} is formally a fibre product $\scrM_{bubble} \times_L \scrM_{Floer}$ of a space of Floer strips and a space of disc bubbles, where  the fibre product is taken over the (presumed smooth and transverse) evaluation maps into $L$.  Moduli spaces of  unparametrised discs on $L$ in some class $\beta$ have virtual dimension $n+\mu_L(\beta)-3 \geq n+1$, hence after evaluation are trivial as chains on $L$, i.e. are  boundaries in $C^*_{sing}(L)$.  That means that, algebraically, degenerations with disc bubbles do not appear when computing $d^2$ in $CF^*(K,L)$.   

In the special case of monotone Lagrangians of minimal Maslov index 2, there is a slight but important difference:  Maslov 2 disks come in $(n-1)$-dimensional moduli spaces, and after evaluation at a boundary marked point disks in class $\beta$ sweep $L$ with some multiplicity $n_{\beta}$.  This leads to a distinguished ``obstruction term" 
\[
\frak{m}_0(L) = \sum_{\beta} n_{\beta} q^{\langle \omega, \beta\rangle} \in \bC[\![q]\!].
\] One can still define Floer cohomology for pairs $L, L'$ with $\frak{m}_0(L) = \frak{m}_0(L')$, since the relevant additional boundaries can then be cancelled against one another, but for monotone Lagrangians with differing $\frak{m}_0$-values one does \emph{not} have $d^2=0$ in the Floer complex. 

\begin{Example} \label{Ex:CY} If $X$ is a Calabi-Yau surface, then disc bubbles for a Maslov zero Lagrangian have negative virtual dimension $n+\mu-3 = -1$, so there are no bubbles for generic $J$ and one can define $CF^*(L,L)$ as a cochain complex. However, in one-parameter families of Lagrangians, isolated disc bubbles can appear, which means that it is not \emph{a priori} clear that the resulting Floer cohomology is Hamiltonian invariant: the continuation map equation \eqref{Eqn:Continuation} may fail to be a chain map because of bubbling.  See \cite[Section 8]{Seidel:HMSquartic}.\end{Example}

\subsubsection{Cancel bubbles: obstruction theory}  If disk bubbles appear, they introduce unwanted boundaries to one-dimensional moduli spaces which prevent $d^2=0$. One can recover the situation if one can cancel out these extra boundaries ``by hand": in Figure \ref{Fig:Bubble}, one may have the evaluation image $ev(\scrM_{bubble}) \in C^*_{sing}(L)$ being a boundary, even if not trivially so for dimension reasons. This is the origin of \emph{obstruction theory} for Floer cohomology. It does not always apply (the locus in $L$ swept by boundaries of disc bubbles may not be nullhomologous), and when it does it depends on the additional choices of bounding cochains (which complicates Hamiltonian invariance, cf. Example \ref{Ex:CY}), but it provides the most general setting in which Floer theory has been developed, notably in \cite{FO3, FO3:technical}.  The class of situations in which this strategy works is sufficiently broad to be a considerable advance over the others, but the details of the undertaking -- relying in general on virtual perturbation theory -- are somewhat fearsome.

In deference to the general situation, Lagrangians whose Floer cohomology is well-defined are said to be \emph{unobstructed}.

\subsection{Technicalities with the definition: Algebraic hypotheses\label{Sec:AlgebraicHyp}} Even once one has decided what class of Lagrangian submanifolds to allow, there are variations on the additional structures they will carry (the existence of which may impose further constraints).

\subsubsection{Choice of coefficient ring.\label{Sec:Novikov}}

Gromov compactness gives control over limits of sequences of holomorphic curves if one has a uniform area bound, but -- as in Novikov's circle-valued Morse theory -- one can still encounter infinitely many moduli spaces in a given problem (here of connecting strips between a fixed pair of intersection points). If infinitely many moduli spaces arise, then  for the sum \eqref{Eqn:FloerDifferential} to make sense one needs to keep track of homotopy classes. Thus, the most generally applicable definition of $HF^*$ involves working over a Novikov field
\begin{equation} \label{Eqn:NovikovField}
\Lambda \ = \ \left\{ \sum a_i q^{\lambda_i} \ | \ a_i \in \bK, \lambda_i \in \bR, \lambda_i \rightarrow \infty \right\}
\end{equation}
based over an underlying coefficient field $\bK$ (if $\bK$ is algebraically closed, so is $\Lambda$).  One then defines the Floer differential \eqref{Eqn:FloerDifferential} to be 
\begin{equation} \label{Eqn:FloerDifferential-Novikov}
d_J(x_+) = \sum_{\beta, x_-} n^{\beta}_{x_+, x_-} q^{\langle \omega, \beta\rangle} x_-
\end{equation}
 with $\beta$ the homotopy class of connecting bigon.   There is a subring $\Lambda_+ \subset \Lambda$ comprising the power series for which the lowest power of $q$ is $\geq 0$.
Floer cohomology is invariant under Hamiltonian isotopy only if one works over the Novikov field; nonetheless, it holds non-trivial information over $\Lambda_+$. The classification theorem for finitely generated $\Lambda_+$-modules shows that 
\[
HF(L,L') \cong \Lambda_+^r \, \bigoplus \ \oplus_{i=1}^n \Lambda_+ / \langle q^{k_i} \rangle
\]
The torsion part obviously dies when passing to $\Lambda$, whilst the free summands persist;  on the other hand, an explicit Hamiltonian isotopy can only change the coefficients $k_i$ by an amount bounded by the Hofer norm of the isotopy, which is information lost when working over the Novikov field. This is already apparent for the most trivial case of the circle $S^1 \subset \bC$.  This is Hamiltonian displaceable, so over $\Lambda$ its self-Floer cohomology vanishes, whilst over the ring $\Lambda_+$ the self-Floer cohomology is the torsion group $\Lambda_+ / \langle q^A \rangle$ where $A$ is the area enclosed by the circle.  Thus, Floer cohomology over $\Lambda_+$  remembers the (in this case elementary) fact that any displacing Hamiltonian must use energy at least $A$.

\begin{Example}
If $L, L'$ are exact with $\theta|_L = df$ and $\theta|_{L'} = df'$ then the area of any connecting holomorphic strip between intersection points $x,y \in L\pitchfork L'$ is given by $f(x)-f'(y)$. The uniform area bound and Gromov compactness imply that only finitely many homotopy classes $\beta$ can be realised by holomorphic strips, so the sum \eqref{Eqn:FloerDifferential} is finite.
\end{Example}

\subsubsection{Spin structures and twisted theories}

So far we have not mentioned the characteristic
of the field $\bK$. If this is not equal to 2, then the counts in \eqref{Eqn:FloerDifferential} or \eqref{Eqn:FloerDifferential-Novikov} need to be \emph{signed} counts of 0-dimensional moduli spaces, so those moduli spaces should be oriented.  A basic fact \cite{deSilva,FO3} is that a choice\footnote{One can get by with $Pin$, but we will stick to $Spin$.} of $Spin$ structure on $L$ induces an orientation on the moduli space $\scrM(\beta)$ of holomorphic discs on $L$ in some class $\beta$; a loop $\gamma: S^1 \rightarrow \scrM(\beta)$ induces, by taking boundary values, a map $\Gamma_{\gamma}: T^2 \rightarrow L$, and $\langle w_1(\scrM(\beta)), \gamma\rangle = \langle w_2(L), \Gamma_{\gamma}\rangle$.  By the same token, if $L$ is not $Spin$, one can't expect moduli spaces of discs or Floer trajectories to be orientable, and one should work in characteristic 2.

The existence of $Spin$ structures on the Lagrangians means that moduli spaces admit orientations. Those orientations are still not canonical, however: one can twist the sign with which any holomorphic disc counts to the Floer differential by its intersection number with an ambient codimension two cycle in $X$ disjoint from the Lagrangians, and such twists are compatible with the breaking of curves and resulting compactifications of moduli spaces which define Floer cohomology (and the operations thereon encountered later).  The upshot is that there is really a version of Floer cohomology $HF^*(L,L';\beta)$ for any choice of ``background" class $\beta \in H^2(X;\bZ_2)$, where the Lagrangians should be relatively spin in the sense that $w_2(L) + \beta|_L = 0$.

\subsubsection{Gradings}  Unlike the chain complexes of classical algebraic topology, the Floer complex is not generally $\bZ$-graded.  At least if the Lagrangian submanifolds $L, L'$ are oriented, then transverse intersection points are positive or negative, and that induces a $\bZ_2$-grading on the complex (if the Lagrangians are orientable but not oriented, this still persists as a relative $\bZ_2$-grading). One can refine this to a $\bZ/N$ grading, where $N$ is the minimal Maslov number of the Lagrangians. In particular, if $L$ and $L'$ have vanishing Maslov class (which necessitates in particular that $2c_1(X)=0$) then the Floer complex can be $\bZ$-graded.

\subsection{Applications, I} \label{Sec:Appl1} As we indicated initially, if $HF^*(L,L') \neq 0$ then $L$ and $L'$ cannot be displaced from one another by Hamiltonian isotopy.  Several of the first applications were variations on this theme: the following are taken from  \cite{Lalonde-Sikorav, FSS2, Polterovich}.

\subsubsection{Nearby Lagrangians} If $L\subset M$ is Lagrangian, $M$ is locally modelled on the cotangent bundle $(T^*L, d\theta)$ with its canonical symplectic structure.  That means the classification problem for Lagrangian embeddings is always at least as complicated as the local problem of classifying Lagrangians in the cotangent bundle. Arnol'd conjectured, cf. Question \ref{Q:Arnold}, that if $K\subset T^*L$ is a closed \emph{exact} Lagrangian it should be Hamiltonian isotopic to the zero-section. In particular
\begin{itemize}
\item  $HF^*(K,L) \cong HF^*(L,L) = H^*(L)$ should be non-zero, so $K$ should intersect the zero-section $L\subset T^*L$;
\item  $K \subset T^*L \rightarrow L$ should be a degree one map and hence onto.
\end{itemize}  

There is a rescaling action of $T^*L$ which multiplies the fibres by some constant $c$; the rescaled submanifolds $cK$ are all exact, hence Hamiltonian isotopic.   Letting $c\rightarrow \infty$, the submanifold $cK$ approximates a union of fibres $\bigcup_{x \in K \cap L} T_x^*L$.  In this schematic, holomorphic strips between $L$ and $cK$ either stay close to a single cotangent fibre, or sweep out a large piece of the cotangent bundle.  Filtering holomorphic discs by area, there is then a spectral sequence, cf. Remark \ref{Rem:Filtration}:
\begin{equation} \label{Eqn:FibresShouldGenerate}
\bigoplus_{x\in K\cap L} \, HF^*(K, T_x^*L) \ \Longrightarrow \ HF^*(K,cK) = H^*(K)
\end{equation}
(the final identity following from Theorem \ref{Lem:Floer}). 
One concludes $L\cap K \neq \emptyset$, and furthermore that $HF^*(K, T_x^*L) \neq 0$ for $x\in L\cap K$.  However, the group $HF^*(K, T_x^*L)$ is independent of the choice of $x$ by a suitable continuation argument, which gives the second part of the second statement: projection $K \rightarrow L$ is onto. This argument is \emph{not} sufficient as it stands to conclude that the projection map has degree 1; we will derive that later, in Section \ref{Sec:NearbyLag}.
 
 \subsubsection{Mapping class groups}  Recall from Section \ref{Sec:AlgFamily} the quadratic map $\pi: \bC^{n+1} \rightarrow \bC$, $(z_1,\ldots,z_{n+1}) \mapsto \sum z_j^2$. The generic fibre $\pi^{-1}(1) \subset (\bC^{n+1}, \omega_{\st})$ is symplectomorphic to $(T^*S^n, d\theta)$. The monodromy of the fibration is isotopic to a compactly supported symplectomorphism, the \emph{Dehn twist} $\tau \in \pi_0\Symp_{ct}(T^*S^n)$, which can be defined more directly as follows. The geodesic flow for the round metric on $S^n$ is periodic with period $\pi$.    Take a Hamiltonian function $H:T^*S^n \rightarrow \bR$ which is a function of $r=|p|$, the norm of the fibre co-ordinate, and which vanishes for $|p| \ll 1$  and has slope $\pi$ for $|p| \geq 1$.    Then one defines $\tau$ to be the (by construction compactly supported) symplectomorphism of $T^*S^n$ given by the composition (antipodal)$\circ \phi_H^1$.  
 
 We claim that $\tau$ is an element of infinite order in $\pi_0\Symp_{ct}(T^*S^n)$, for $n\geq 1$.  (For odd $n$ this is elementary, following from homology considerations, but in even dimensions the map has finite order in the smooth mapping class group.)  
   Consider the Floer cohomology $HF^*(\tau^m (T_x^*), T_y^*)$, with $x, y$ distinct but not antipodal.  Although these are non-compact Lagrangians, the intersections live in a compact set, and the Lagrangians are exact and modelled on Legendrian cones at infinity, which means that $HF$ is defined unproblematically (as an alternative, one could cap off the relevant two cotangent fibres to obtain a space containing an $A_3$-chain of Lagrangian spheres; compare to Section \ref{Sec:MCGIII}).  Moreover, since $c_1(T^*S^n) = 0$ and the Lagrangians have trivial Maslov class, $HF$ is also $\bZ$-graded.  The Floer complex generators correspond to geodesics which wrap at most $m$ times around the appropriate great circle. The $\bZ$-grading in the Floer complex reproduces the Morse grading of the geodesics as critical points of the energy function on loop space (determined by numbers of Jacobi fields), and one finds \cite{FrauSchlenk} that $CF^*(\tau^m (T_x^*), T_y^*)$ is concentrated in degrees
 \[
 (2n-2) \{0,1,...,(m-1)\}, \ \mathrm{and} \  (3n-3)\{0,1,...,(m-1)\}.
 \]
 If $n \geq 3$,  the groups are never non-zero in adjacent degrees, the Floer differential vanishes, and the rank of the Floer group grows with $m$.   The result also holds when  $n=2$ but an extra step is needed: for instance, one can work equivariantly with respect to a symplectic involution to reduce to the easier case of $T^*S^1$ (in this fixed locus there are no Floer differentials for homotopy reasons, meaning one can achieve equivariant transversality, cf. Example \ref{Ex:TransversalityFiniteGroup}).

\subsection{Algebraic structures} Floer cochain groups have a wealth of additional structure; we begin with the (chain-level) constructions which do descend to structures in cohomology.

\subsubsection{Product structures}
Given three Lagrangians $L,M,N \subset X$, there is a degree zero product 
\[
\mu^2: HF^*(M,N) \otimes HF^*(L,M) \longrightarrow HF^*(L,N)
\]
via the holomorphic triangle product schematically indicated below.
\begin{center}
\includegraphics[scale=0.4]{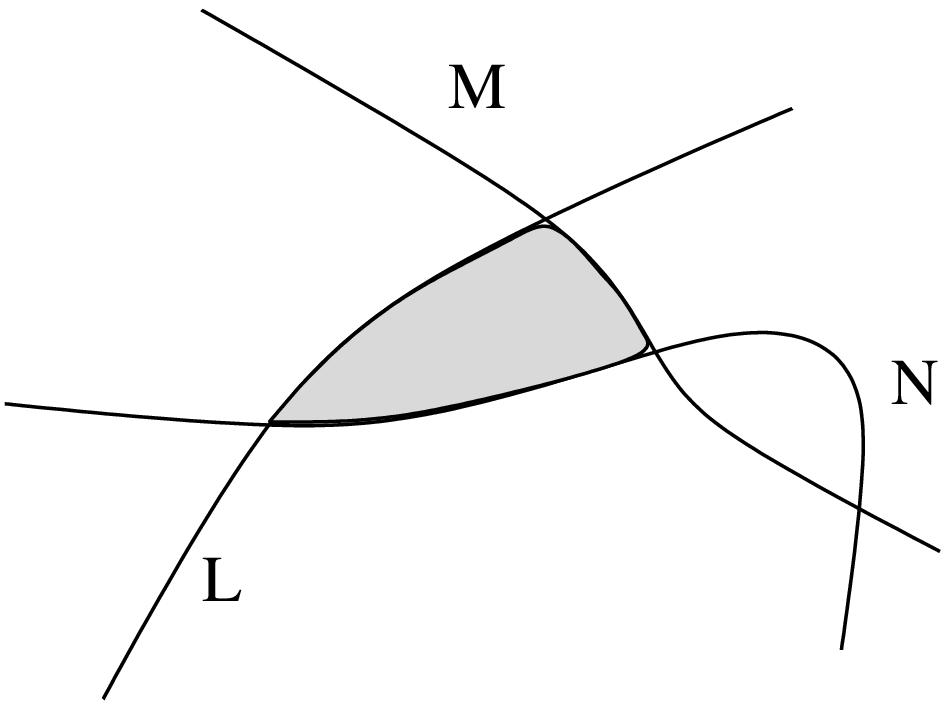}
\end{center}
The figure indicates that given input (presumed transverse) intersection points in $L\cap M$ and $M \cap N$, the coefficient of some given output point in $L\cap N$ is given by a count of triangles as indicated, subject to the usual caveats: one counts elements of zero-dimensional moduli spaces, under suitable geometric hypotheses to ensure compactness, and taking Novikov coefficients if infinitely many distinct moduli spaces can in principle contribute.  
\begin{center}
\begin{figure}[ht]
\includegraphics[scale=0.45]{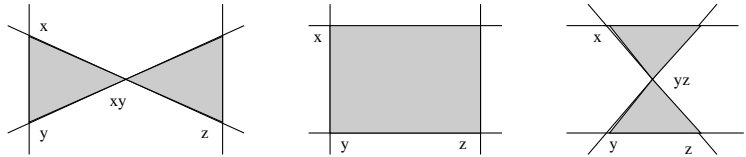}
\caption{The triangle product is associative on cohomology.\label{Fig:BrokenTriangles}}
\end{figure}
\end{center}

The fact that the chain-level operation descends to cohomology reflects the fact that, modulo bubbling, the boundary of a one-dimensional moduli space of such triangles has strata given by  broken configurations which involve breaking ``at the corners", and count the coefficient in a fixed output intersection point  of  the expressions $\mu^2(d(x),y)$, $\mu^2(x,d(y))$ and $d\mu^2(x,y)$, with $d$ the Floer differential.  Similarly, the fact that one-dimensional moduli spaces of quadrilaterals can break as shown in Figure \ref{Fig:BrokenTriangles} shows that the triangle product is associative on cohomology: if all the inputs are Floer cycles, $d(\cdot) = 0$, only the two ``concatenated triangle" pictures define non-trivial strata to the boundary of the one-dimensional moduli space, and those correspond to the coefficient of the top right  intersection point in each of $\mu^2(x,\mu^2(y,z))$ and $\mu^2(\mu^2(x,y),z)$.

It follows that for a single Lagrangian $L$, $HF^*(L,L)$ is naturally a ring, which moreover is unital.  The ring $HF^*(L,L)$ has Poincar\'e duality, arising from reversing the roles of the two boundary conditions for holomorphic strips. Indeed, there is a trace $\tr: HF^{n}(L,L) \rightarrow \bK$ which, together with the triangle product, gives rise to a non-degenerate pairing
\[
HF^*(K,L) \otimes HF^{n-*}(L,K) \stackrel{\tr}{\longrightarrow} \bK.
\]
 The following Proposition is essentially due to Oh \cite{Oh:Thomas} (the compatibility with the ring structure was made explicit in \cite{Buhovski:ring}), see also \cite{Albers,BiranCornea}. It provides a key computational tool in the subject.

\begin{Proposition}\label{Lem:Oh}
There is a  spectral sequence $H^*(L;\Lambda) \Rightarrow HF^*(L,L)$ which  is compatible with the ring structures in ordinary and Floer cohomology. 

\end{Proposition}

\begin{proof}[Sketch]
Fix a Morse function $f: L \rightarrow \bR$. As in Theorem \ref{Lem:Floer}, for a suitable $J$ Morse flow-lines define Floer trajectories. The spectral sequence arises by filtering the Floer differential by action, such that all low-energy solutions arise from Morse trajectories.  In the monotone case, where action and index correlate,  the spectral sequence arises from writing the Floer differential as a sum $\partial_0 + \partial_1 + \cdots + \partial_k$ where $\partial_0$ is the Morse differential in $C^*_{Morse}(L)$ and $\partial_i$ has degree $1-iN_L$, where $N_L$ is the minimal Maslov number of $L$.  More vividly, there is a Morse-Bott (or ``pearl") model for Floer cohomology \cite{BiranCornea}, in which the Floer differential counts gradient lines of $f$ interrupted with holomorphic disks with boundary on $L$, cf. Figure \ref{Fig:PearlyGradient}, and $\partial_i$ considers configurations where the total Maslov index of the discs is $iN_L$.  The Morse product counts $Y$-trees of flow-lines, and this has a generalisation to $Y$-shaped graphs of pearls, which gives the compatibility of the spectral sequence with ring structures. \end{proof}
\begin{center}
\begin{figure}[ht]
\includegraphics[scale=0.4]{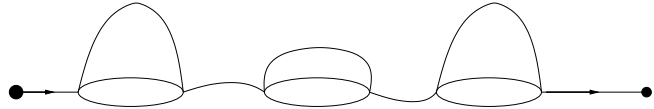}
\caption{A Morse-Bott or ``pearly" Floer trajectory\label{Fig:PearlyGradient}}
\end{figure}
\end{center}

\begin{Example}
The ring $HF^*(L,L)$ is \emph{not} always skew-commutative. 
Let $L = S^1_{eq} \subset \bP^1$; then $HF^*(L,L) \cong \bC[x]/(x^2=1)$ has  $\mu^2(x,x)=1_L$ with $|x|=1$ the top degree generator. Taking a pearl model, this product is defined by evaluation at the output of  discs $(D^2, \partial D^2) \rightarrow (\bP^1, S^1_{eq})$  with two input marked points which are fixed generic points on the equator. The 3 boundary marked points have a fixed cyclic order, which leads to two possibilities: if the image disc is the upper hemisphere, evaluation sweeps one chain on $S^1_{eq}$, whilst if it is the lower hemisphere it sweeps the complementary chain. Together these make up the fundamental class of $L$.
\end{Example}

\begin{Example} \label{Ex:ToricFibre} Let $X$ be monotone and $L\subset X$ a monotone Lagrangian torus of minimal Maslov number $\mu_L$. The  differential $\partial_1$ in the spectral sequence $H^*(L) \Rightarrow HF^*(L,L)$ is a map
\begin{equation} \label{Eqn:ToricDifferential}
H^1(L) \longrightarrow H^{2-\mu_L}(L).
\end{equation}
If $\mu_L=2$ and this is non-zero then the unit in $H^0(L)$ becomes a boundary and vanishes in $HF^*(L,L)$, which implies the latter group is identically zero. If \eqref{Eqn:ToricDifferential} vanishes then $H^1(L) \subset \ker(\partial_1)$, and since that kernel is a subring, $H^*(L) \subset \ker(\partial_1)$ and $\partial_1 \equiv 0$.  Since the higher $\partial_i$ have degree $1-2i \leq -2$, the same argument shows that $H^1(L)$, whence $H^*(L)$, lies in the kernel of every $\partial_i$, which  implies that the spectral sequence degenerates.  Therefore, if $L\subset X$ is a \emph{displaceable} monotone torus (e.g. one in $\bC^n$), then $\mu_L = 2$  since $HF(L,L)=0$.
\end{Example}

There is a category Don$(X)$, the Donaldson category, whose objects are (unobstructed) Lagrangian submanifolds, and whose morphisms are Floer cohomology groups. The composition is given by the holomorphic triangle product.  If $X$ is monotone, this category splits up into orthogonal subcategories
\begin{equation} \label{Eqn:DonSplits}
\mathrm{Don}(X) = \bigoplus_{\lambda} \mathrm{Don}(X;\lambda)
\end{equation}
where $\lambda$ varies through the possible values $\frak{m}_0(L)$, since $HF^*(L,L')$ is only well-defined for pairs of Lagrangians when $\frak{m}_0(L) = \frak{m}_0(L')$. Since $\frak{m}_0(L)$ counts Maslov 2 discs, if $X$ has minimal Chern number $>1$ then any simply-connected Lagrangian necessarily lies in  Don$(X;0)$.

\subsubsection{Module structures}\label{Subsec:ModuleStructures}
The Floer cohomology $HF^*(L, L')$ is a right module for the unital ring $HF^*(L,L)$.   There are also module actions over quantum cohomology, so for any (orientable) Lagrangian $L$ there is a unital ring homomorphism 
\begin{equation} \label{Eqn:ClosedOpen}
\mathcal{CO}^0: QH^*(X) \longrightarrow HF^*(L,L)
\end{equation}
 called (the length zero part of) the closed-open  string map. This is defined by counting holomorphic disks with boundary on $L$, an interior input marked point constrained on a cycle in $X$ representing a $QH^*$-generator, and a boundary marked point at which we evaluate to define the output Floer class. $QH^*(X)$ is commutative, and $\mathcal{CO}^0$ has image in the centre of $HF^*(L,L)$.

\begin{Example} The quantum cohomology of $\bC\bP^n$ is $\Lambda[x]/(x^{n+1}=1)$, with $|x|=2$.  Suppose $L\subset \bC\bP^n$ is a Lagrangian submanifold. If $HF^*(L,L)$ is well-defined, since $x$ maps in $HF^*(L,L)$ to an invertible element, the latter group must be 2-periodic. 
\end{Example}

A beautiful observation of Auroux, Kontsevich and Seidel \cite{Auroux}  is that in the monotone case, and working away from characteristic 2,  the map \eqref{Eqn:ClosedOpen} takes $c_1(X) \mapsto \frak{m}_0(L) e_L$, which implies that the values $\lambda$ occuring in \eqref{Eqn:DonSplits} are exactly the eigenvalues of the endomorphism $\ast c_1(X)$ acting on $QH^*(X)$.  See Remark \ref{Rem:StarC1} for a brief discussion.

\subsubsection{Exact triangles\label{Subsec:ExactTriDon}} Let's say, provisionally, that a triple of Lagrangian submanifolds and morphisms 
\[
\xymatrix{ L \ar^{a}[r] & L' \ar^{b}[r] & \ar@/^1.5pc/[ll]_{[1]} L''}
\]
  form an \emph{exact triangle} in Don$(X)$ if for every other Lagrangian $K$, composition with the given morphisms (Floer cocycles) induces a long exact sequence of Floer cohomology groups
\[
\cdots  \rightarrow HF^i(K,L) \stackrel{a}{\longrightarrow} HF^i(K,L')\stackrel{b}{\longrightarrow} HF^i(K,L'') \rightarrow HF^{i+1}(K,L) \rightarrow \cdots.
 \]
 Perhaps surprisingly, exact triangles exist:

\hspace{1em} $\bullet$ If $L \pitchfork L' = \{p\}$ meet transversely in a single point, and $L'' = L\# L'$ is the Lagrange surgery \cite{Polterovich} (which is topologically the connect sum, replacing a neighbourhood $\bR^n \cup i\bR^n$ of the double point by a Lagrangian one-handle $S^{n-1} \times \bR$), then \cite{BiranCornea2, FO3}  there is an exact triangle $L \stackrel{p}{\longrightarrow} L' \longrightarrow L''$.  In Figure \ref{Fig:SurgeryTriangle}, the intersection point of the straight Lagrangians $L, L'$ is resolved going from left to right; the dotted arc is a third Lagrangian $K$; and one sees two holomorphic discs contributing to the complex $CF^*(K, L\#L')$ on the right, one coming from  the complex $CF^*(K,L)$, and one coming from multiplication with $p$ viewed as a map $CF(K,L) \rightarrow CF(K,L')$, exhibiting schematically the  quasi-isomorphism 
\[
CF(K,L\#L') \stackrel{\sim}{\longrightarrow} \left( CF(K,L) \oplus CF(K,L'), \left(\begin{array}{cc} \partial_L & 0 \\ \mu^2(p,\cdot) & \partial_{L'} \end{array} \right)\right)
\]
(in higher dimensions, the corresponding injection from the set of rigid holomorphic triangles to rigid  discs with boundary on $K$ and $L\#L'$ is non-trivial).
\begin{center}
\begin{figure}[ht]
\includegraphics[scale=0.4]{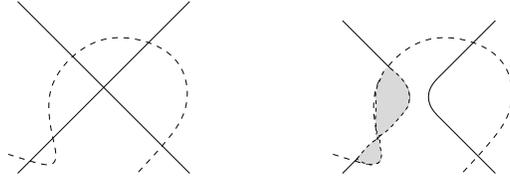}
\caption{Towards the surgery exact triangle\label{Fig:SurgeryTriangle}}
\end{figure}
\end{center}

\newcommand{\ev}{\mathrm{ev}}
\hspace{1em} $\bullet$ If $V\cong S^n \subset X$ is a (monotone or exact) Lagrangian sphere, and $\tau_V$ denotes the Dehn twist in $V$, then \cite{Seidel:LES} there is an exact triangle  $L \stackrel{\phi}{\longrightarrow} \tau_V(L) \rightarrow HF^*(V,L)\otimes V$ in the additive enlargement of Don$(X)$, in which we formally allow direct sums of objects, or equivalently tensor products of Lagrangians by graded vector spaces.  Fix any other Lagrangian $K$. One starts by identifying intersection points of $K \cap \tau_V(L)$ with points of $K \cap L$ (which we assume lie outside the support of $\tau_V$) together with a copy of $K \cap V$ for each point of $V \cap L$, as $\tau_V$ wraps $L$ around $V$ each time $L$ meets $V$. That gives 
$CF(K,\tau_V(L))  = CF(K,L) \oplus (CF(V,L)\otimes CF(K,V))$ as vector spaces. The morphism $\phi$ counts sections of a Lefschetz fibration over $D^2\backslash \{-1\}$ with one critical point at $0$ (where the vanishing cycle $V$ has been collapsed to a point) and boundary conditions $L, \tau_V(L)$.  The required nullhomotopy 
\begin{equation} \label{Eqn:Nullhtpy}
\kappa: CF(V,L) \rightarrow CF(V,\tau_V(L)), \quad \partial \kappa \pm \kappa \partial + \mu^2(\phi, \cdot) = 0
\end{equation}
which yields the LES is obtained from a 1-parameter family of Lefschetz fibrations as in Figure \ref{Fig:LESnullhtpy}. On the left one sees the composite $\mu^2(\phi,\cdot)$; the final fibration on the right of the Figure admits no rigid sections, because there are \emph{no} such sections of the model Lefschetz fibration $\bC^{n+1} \rightarrow \bC$ with boundary condition the vanishing cycle $V$ which has  bubbled off, since all moduli spaces of sections of the model have odd virtual dimension.  Thus if $\kappa$ counts rigid curves in the whole family, it yields \eqref{Eqn:Nullhtpy}.  Viewing $\kappa$ as a map $CF(V,L)\otimes V \rightarrow \tau_V(L)$, and letting $\ev: CF(V,L) \otimes V \rightarrow L$ be the tautological evaluation, \eqref{Eqn:Nullhtpy} shows $\mu^2(\phi,\ev)$ is nullhomotopic via $\kappa$, and identifies Cone$(\ev)$ with $\tau_V(L)$ as required.
\begin{center}
\begin{figure}[ht]
\includegraphics[scale=0.4]{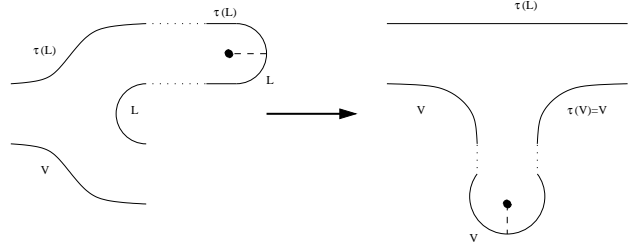}
\caption{Towards the Dehn twist exact triangle\label{Fig:LESnullhtpy}}
\end{figure}
\end{center}

\subsubsection{Functors from correspondences}\label{Sec:Quilt}  Let Don$^{\#}(Y)$ denote the category of generalised Lagrangian submanifolds of $Y$, whose objects are sequences of Lagrangian correspondences
\[
\{pt\} \stackrel{L_0}{\longrightarrow} Y_0 \stackrel{L_{01}}{\longrightarrow} \cdots \stackrel{L_{k-1,k}}{\longrightarrow} Y_k = Y
\]
with $L_{i,i+1} \subset Y_i^- \times Y_{i+1}$ a Lagrangian submanifold.  Here $(Y,\omega_Y)^- = (Y, -\omega_Y)$. There is an (essentially formal) extension of Floer cohomology to pairs of generalised Lagrangians, due to Wehrheim and Woodward \cite{WW1, WW2}. For instance, given $L \subset X$, $K\subset Y$ and $\Gamma \subset Y^- \times X$, $HF(L, (\{pt\} \stackrel{K}{\longrightarrow} Y \stackrel{\Gamma}{\longrightarrow} X))$ is by definition  $HF(K \times L, \Gamma)$, computed in $Y^- \times X$. There is a more involved extension of the Floer product  to generalised Lagrangians, involving counts of ``quilted" Riemann surfaces (ones with interior seams labelled by Lagrangian boundary conditions).   Let $(X,\omega_X)$ and $(Y,\omega_Y)$ be symplectic manifolds and $\Gamma \subset X^-\times Y$ a Lagrangian correspondence, where $X^-$ denotes $(X,-\omega)$. Under strong monotonicity hypotheses, $\Gamma$ induces a functor
\[
\Phi_{\Gamma}: \mathrm{Don}(X) \longrightarrow \mathrm{Don}^{\#}(Y),  \qquad  \Phi_{\Gamma}: L \, \mapsto \,  \left( \{pt\} \stackrel{L}{\longrightarrow} X \stackrel{\Gamma}{\longrightarrow} Y\right)
\]
which on objects acts by the indicated concatenation.
An important and difficult theorem due to Wehrheim and Woodward is that if the geometric composition 
\[
\Gamma \circ L \ = \ \left\{ u \in Y \ \big| \ (x,u) \in \Gamma \ \mathrm{for \ some} \ x\in L\, \right\}
\]
is an embedded Lagrangian submanifold, then $\Phi_{\Gamma}(L)$ and $\Gamma\circ L$ are isomorphic in Don$^{\#}(Y)$ (in particular, the formal composition is isomorphic to an object of the usual Donaldson category). 
There is an underlying functor, built using quilt theory, 
\begin{equation} \label{Eqn:QuiltFunctor}
\Phi: \mathrm{Don}(X^- \times Y) \longrightarrow \mathrm{Fun}(\mathrm{Don}(X), \mathrm{Don}^{\#}(Y))
\end{equation}
where $\Phi: \Gamma \mapsto \Phi_{\Gamma}$. In particular, if $HF(\Gamma,\Gamma) = 0$ (computed in $X^-\times Y$) then $\Phi_{\Gamma}$ is the zero functor, whilst if $\Phi_{\Gamma}(L) \neq 0$ then $\Gamma \neq 0$.  

\begin{Example} \label{Ex:Corr} cf. \cite{WW2}.
Let $S^3 \subset \bC^2$ denote the unit sphere. The Hopf map $h: S^3 \rightarrow S^2$ has $(\omega_{\st})|_{S^3} = h^*(\omega_{\operatorname{FS}})$. The graph of the composition $(\mathrm{antipodal})\circ h$ defines a Lagrangian submanifold of $B^4(r) \times \bP^1$, for any $r>1$, and hence a Lagrangian submanifold $\Gamma=S^3 \subset \bP^2 \times \bP^1$ with the monotone symplectic form.   View $\Gamma$ as a functor from $\mathrm{Don}(\bP^1)$ to $\mathrm{Don}^{\#}(\bP^2)$.  The image $\Phi_{\Gamma}(S^1_{eq})$ is quasi-isomorphic to the Clifford torus in $\bP^2$. This has non-trivial self-Floer cohomology as an instance of Example \ref{Ex:ToricFibre}, which implies $HF(\Gamma,\Gamma) \neq 0$. 
\end{Example}

At first sight the passage to Don$^{\#}(Y)$ is slightly intimidating, since its definition involves arbitrary sequences of intermediate symplectic manifolds, rendering it apparently unmanageable.  However, given a finite collection of Lagrangians $\{L_i\} \subset Y$ and associated algebra $A = \oplus_{i,j} HF^*(L_i, L_j)$, there is a functor Don$^{\#}(Y) \rightarrow \mathrm{mod}$-$(A)$, which takes any generalised Lagrangian to the direct sum of its Floer cohomologies with the $L_i$, and  this enables one to bring tools from non-commutative algebra to bear on the extended category.

\subsection{Applications, II} We briefly illustrate how the algebraic structures introduced so far increase the applicability of the theory; the examples come from \cite{Buhovsky} and \cite{Smith:HFQuadrics}.

\subsubsection{Nearby Lagrangians}  Let $L \subset T^*S^{2k+1}$ be an exact Lagrangian with $H_1(L;\bZ) = 0$. We claim that $H^*(L) \cong H^*(S^{2k+1})$.  As in Example \ref{Ex:Corr}, the Hopf map
\[
\bC^{k+1}\backslash \{0\} \supset S^{2k+1} \stackrel{h}{\longrightarrow} \bP^k
\]
given by projectivising has the feature that $h^*\omega_{\operatorname{FS}} = (\omega_{\st})|_{S^{2k+1}}$. Therefore, taking the graph, there is a Lagrangian embedding $S^{2k+1} \subset \bC^{k+1} \times (\bP^k)^-$.  Rescaling $L\subset T^*S^{2k+1}$ to lie near the 0-section if necessary, $L$ also embeds in this product.  Since by hypothesis $H_1(L;\bZ) = 0$, the embedding is monotone, hence $HF(L,L)$ is well-defined, computing in the product $\bC^{k+1} \times (\bP^k)^-$. By translations in the first factor it is obviously displaceable off itself, so $HF(L,L)=0$. 

Now consider the Oh spectral sequence, Proposition \ref{Lem:Oh}.  $L$ has vanishing Maslov class in $T^*S^{2k+1}$, but  has minimal Maslov number $2k+2$ in $\bC^{k+1} \times (\bP^k)^-$. That means the only non-trivial differential in the Oh spectral sequence has the form
\[
H^{i+2k+1}(L) \rightarrow H^i(L) \rightarrow H^{i-2k-1}(L) \qquad \forall \ 0 \leq i \leq 2k+1.
\]
The desired conclusion is an immediate consequence.

\subsubsection{Mapping class groups}   Let $S$ be a closed surface of genus 2, $p
\in S$, and $\scrM(S)$ the moduli space of conjugacy classes of representations of $\pi_1(S\backslash \{p\})$ into $SU(2)$ which take the boundary loop to $-I$. This is symplectic, being a  quotient of the infinite-dimensional space of connections  in a rank two bundle $E \rightarrow S$, which has a natural symplectic form $\omega(a,b) = \int_S \tr(a \wedge b)$ for $a,b \in \Omega^1(S;\End(E))$.  The  mapping class group $\Gamma(S,p)$ of diffeomorphisms of $S$ fixing $p$ acts symplectically on $\scrM(S)$, and we claim that for any $\phi \in \Gamma(S,p)$, $HF^*(\phi) \neq \{0\}$.  

The heart of the proof is the following inductive argument.  A homologically essential  simple closed curve $\gamma \subset S$ defines a Lagrangian 3-sphere $SU(2) \cong L_\gamma \subset \scrM(S)$ by taking the connections with trivial holonomy along $\gamma$. The Dehn twist $\tau_\gamma \in \Gamma(S,p)$ acts on $\scrM(S)$ by the higher-dimensional Dehn twist $\tau_{L_{\gamma}}$.  The space $\scrM(S)$ is monotone with minimal Chern number 2, and Spec$(\ast c_1: QH^*(\scrM(S)) \rightarrow QH^*(\scrM(S)))$ is $\{0, \pm 4i\}$.  There are exact triangles
\begin{equation} \label{Eqn:InductiveTriangle}
\xymatrix{ HF^*(\phi) \ar[r] & HF^*(\tau_{L_\gamma} \circ \phi)  \ar[r] & \ar@/^1.5pc/[ll]^{[1]} HF^*(L_{\gamma}, \phi(L_{\gamma}))}
\end{equation}
for any $\phi \in \Gamma(S,p)$ and essential $\gamma \subset S$.  Consider the corresponding exact triangle of generalised eigenspaces for the eigenvalue $4i$.  Since $c_1 \mapsto \frak{m}_0(L_{\gamma}) \in HF^*(L_{\gamma}, L_{\gamma}) = A$, and $L_{\gamma}$ has minimal Maslov number $>2$, $c_1$ acts trivially on $A$, and hence on the $A$-module $HF^*(L_{\gamma}, \phi(L_{\gamma}))$. It follows that the $\ast c_1$-generalised eigenspace for the eigenvalue $4i$ is unchanged by composition with $\tau_{L_{\gamma}}$, and an analogous argument shows it is unchanged by composing with an inverse Dehn twist.  Since such twists generate the mapping class group, inductively it suffices to prove that the relevant eigenspace is non-zero for $HF^*(\id)\cong QH^*(\scrM(S))$. A theorem of Newstead \cite{Newstead} identifies $\scrM(S)$ with an intersection of two quadric hypersurfaces in $\bP^5$, from which the computation of $QH^*$ is straightforward \cite{SKD}: the $4i$-generalised eigenspace has rank 1.

A corollary is that if $Y \rightarrow S^1$ is any 3-manifold fibred by genus 2 curves, then $\pi_1(Y)$ admits a non-abelian $SO(3)$-representation.

\subsection{Shortcoming} \label{Subsec:Shortcoming} The following situation often arises: a given symplectic manifold $X$ contains some finite collection of ``distinguished" Lagrangian submanifolds, which one expects to generate all Lagrangians by iteratedly forming cones of morphisms (e.g. by  Lagrange surgeries, or taking images under Dehn twists).   The finite collection of Lagrangians $\{L_i\}$ define a ring 
\[
A = \oplus_{i,j} HF^*(L_i, L_j)
\]
and any other Lagrangian $K$ defines a right $A$-module
\[
Y_K = \oplus_i HF^*(L_i, K).
\]
Now given two Lagrangians $K, K'$ there are two obvious groups one may consider: the geometrically defined Floer cohomology $HF^*(K,K')$, and the algebraically defined group of morphisms $\Ext_{A-mod}^*(Y_K, Y_{K'})$.  With the structure developed thus far, even if the intuition that the $\{L_i\}$ should ``generate" is correct, these two groups will typically disagree.

\begin{center}
\begin{figure}[ht]
\setlength{\unitlength}{1cm}
\begin{picture}(5,2)(0,-1)
\qbezier[200](0,0)(0,1.2)(1.5,1)
\qbezier[200](1.5,1)(2.5,0.85)(3.5,1)
\qbezier[200](3.5,1)(5,1.2)(5,0)
\qbezier[200](0,0)(0,-1.2)(1.5,-1)
\qbezier[200](1.5,-1)(2.5,-0.85)(3.5,-1)
\qbezier[200](3.5,-1)(5,-1.2)(5,0)
\qbezier[60](1,0)(1,0.3)(1.5,0.3)
\qbezier[60](2,0)(2,0.3)(1.5,0.3)
\qbezier[60](1,0)(1,-0.3)(1.5,-0.3)
\qbezier[60](2,0)(2,-0.3)(1.5,-0.3)
\qbezier[60](3,0)(3,0.3)(3.5,0.3)
\qbezier[60](4,0)(4,0.3)(3.5,0.3)
\qbezier[60](3,0)(3,-0.3)(3.5,-0.3)
\qbezier[60](4,0)(4,-0.3)(3.5,-0.3)
\put(1.5,0){\circle{1.1}}
\put(3.5,0){\circle{1.1}}
\qbezier[60](0,0)(0,-0.1)(0.5,-0.1)
\qbezier[60](1,0)(1,-0.1)(0.5,-0.1)
\qbezier[20](0,0)(0,0.1)(0.5,0.1)
\qbezier[20](1,0)(1,0.1)(0.5,0.1)
\qbezier[60](2,0)(2,-0.1)(2.5,-0.1)
\qbezier[60](3,0)(3,-0.1)(2.5,-0.1)
\qbezier[20](2,0)(2,0.1)(2.5,0.1)
\qbezier[20](3,0)(3,0.1)(2.5,0.1)
\qbezier[60](4,0)(4,-0.1)(4.5,-0.1)
\qbezier[60](5,0)(5,-0.1)(4.5,-0.1)
\qbezier[20](4,0)(4,0.1)(4.5,0.1)
\qbezier[20](5,0)(5,0.1)(4.5,0.1)
\qbezier[100](2.5,0.91)(2.6,0.91)(2.6,0)
\qbezier[100](2.5,-0.91)(2.6,-0.91)(2.6,0)
\qbezier[30](2.5,0.91)(2.4,0.91)(2.4,0)
\qbezier[30](2.5,-0.91)(2.4,-0.91)(2.4,0)
\put(-0.35,0){$\zeta_1$}
\put(0.8,0.5){$\zeta_2$}
\put(2.1,-0.35){$\zeta_3$}
\put(3.9,0.5){$\zeta_4$}
\put(5.1,0){$\zeta_5$}
\put(2.65,0.7){$\sigma$}
\end{picture}
\caption{A genus two surface\label{Fig:genus2pencil}}
\end{figure}
\end{center}
\begin{Example} Consider the genus two surface.  
The curves $\zeta_i$ generate all simple closed curves, in the sense that any homotopically essential curve must have non-trivial geometric intersection number with some $\zeta_i$, and the Dehn twists $\tau_{\zeta_j}$ generate the mapping class group\footnote{It will follow from Proposition  \ref{Prop:SpheresGenerate} that the curves $\zeta_i$ split-generate the Fukaya category.}. Over the ring $A = \oplus_{i,j} HF^*(\zeta_i, \zeta_j)$  the module $Y_{\sigma}$ defined by the curve $\sigma$ is isomorphic to the sum $\hat{Y}_{\sigma} = S_{\zeta_3} \oplus S_{\zeta_3}[1]$ of the simple module over $\zeta_3$ and its shift, essentially because there are no non-constant holomorphic lunes or triangles in the picture.  However, the endomorphism algebra of $S_{\zeta_3} \oplus S_{\zeta_3}[1]$ is a matrix algebra, in particular Ext$_{A-\mathrm{mod}}(\hat{Y}_{\sigma}, \hat{Y}_{\sigma}) \not \cong HF^*(\sigma,\sigma) = H^*(S^1)$.  
\end{Example}

The problem here is that $A$ is not just a ring, and in regarding it as such we are ignoring important structure: we introduce that structure next.

\begin{Notes}
For basics on Floer (and quantum) cohomology, see \cite{Floer, Oh, AudinDamian, McD-S2}. For algebraic structures, see \cite{BiranCornea, FO3, WW2, RitterSmith}, and many of Seidel's papers, e.g. \cite{Seidel:graded, Seidel:MCG, Seidel:LES, Seidel:twist}.   
\end{Notes}

\section{The Fukaya category}\label{Sec:Fukaya}

\subsection{$A_{\infty}$-categories}

A (non-unital) $A_{\infty}$-category $\scrA$ over $\bK$ comprises: a set of objects Ob$\,\scrA$; for each $X_0, X_1 \in$ Ob$\,\scrA$ a ($\bZ_2$-graded or $\bZ$-graded, depending on the setting) $\bK$-vector space $hom_{\scrA}(X_0, X_1)$; and $\bK$-linear composition maps, for $k \geq 1$,
\[
\mu_{\scrA}^k: hom_{\scrA}(X_{k-1},X_k) \otimes \cdots \otimes hom_{\scrA}(X_0,X_1) \longrightarrow hom_{\scrA}(X_0,X_k)[2-k]
\]
of degree $2-k$ (the notation $[l]$ refers to \emph{downward} shift by $l \in \bZ$, or $l\in \bZ_2$ as appropriate).  The maps $\{\mu^k\}$ satisfy a hierarchy of quadratic equations
\[
\sum_{m,n} (-1)^{\maltese_n} \mu_{\scrA}^{k-m+1} (a_d,\ldots,a_{n+m+1}, \mu_{\scrA}(a_{n+m},\ldots,a_{n+1}),a_n \ldots, a_1) = 0
\]
with $\maltese_n = \sum_{j=1}^n |a_j|-n$, $|a|$ the degree of $a$, and where the sum runs over all possible compositions: $1 \leq m \leq k$, $0\leq n \leq k-m$.  The equations imply in particular that $hom_{\scrA}(X_0,X_1)$ is a cochain complex with differential $\mu^1_{\scrA}$; the cohomological category $H(\scrA)$ has the same objects as $\scrA$ but morphism groups are the cohomologies of these cochain complexes.  $H(\scrA)$  has an associative composition 
\[
[a_2]\cdot[a_1] = (-1)^{|a_1|}[\mu_{\scrA}^2(a_2,a_1)]
\]
whereas the chain-level composition $\mu_{\scrA}^2$ on $\scrA$  is only associative up to homotopy.  (Thus $\scrA$  is not in fact a category.) 
More explicitly, if $x,y,z$ satisfy $\mu^1(\cdot) = 0$ then in $\scrA$
\[
\mu^1\mu^3(x,y,z) + \mu^2(x,\mu^2(y,z)) + \mu^2(\mu^2(x,y),z) = 0.\]
The higher-order compositions $\mu_{\scrA}^d$ are \emph{not} chain maps, and do not descend to cohomology.   An $A_{\infty}$-category with a single object is an $A_{\infty}$-algebra.

\begin{Example}
The singular chain complex $C_*(\Omega Q)$ of the based loop space of a topological space $Q$ is an $A_{\infty}$-algebra, reflecting the fact that $\Omega Q$ is an $h$-space, with a composition from concatenation of loops which is associative up to systems of higher homotopies. 
\end{Example}

A non-unital $A_{\infty}$-functor $\cF: \scrA \rightarrow \scrB$ between non-unital $A_{\infty}$-categories $\scrA$ and $\scrB$ comprises a map $\cF$: Ob$\,\scrA \rightarrow$ Ob$\,\scrB$, and multilinear maps for $d\geq 1$
\[
\cF^d: hom_{\scrA}(X_{d-1}, X_d) \otimes \cdots \otimes hom_{\scrA}(X_0,X_1) \rightarrow hom_{\scrB}(\cF X_0, \cF X_d)[1-d]
\]
now satisfying the polynomial equations
\begin{multline*}
\sum_{r} \sum_{s_1+\cdots+s_r=d} \mu_{\scrB}^r(\cF^{s_r}(a_d,\ldots,a_{d-s_r+1}),\ldots, \cF^{s_1}(a_{s_1},\ldots, a_{1})) \\
\quad  = \ \sum_{m,n} (-1)^{\maltese_n} \cF^{d-m+1}(a_d,\ldots, a_{n+m+1}, \mu_{\scrA}^m (a_{n+m},\ldots, a_{n+1}), a_{n},\ldots a_1)
\end{multline*}
Any such defines a functor $H(\cF):H(\scrA) \rightarrow H(\scrB)$ which takes $[a] \mapsto [\cF^1(a)]$. We say $\cF$ is a \emph{quasi-isomorphism} if $H(\cF)$ is an isomorphism.   The class of $A_{\infty}$-functors from $\scrA$ to $\scrB$  form the objects of a non-unital $A_{\infty}$-category $\scrQ=nu$-$fun(\scrA,\scrB)$,  see \cite{FCPLT}. 
If one pictures the $A_{\infty}$-operation (with many inputs and one output) as a tree, then the $A_{\infty}$-relations involve summing over images as on the right of the top line of Figure \ref{Fig:Ainftytree}.  In that language, if the functor operation is like the mangrove (partially submerged tree) on the lower left of Figure \ref{Fig:Ainftytree}, the functor relations involve summing over pictures as on the lower right.
\begin{center}
\begin{figure}[ht]
\includegraphics[scale=0.3]{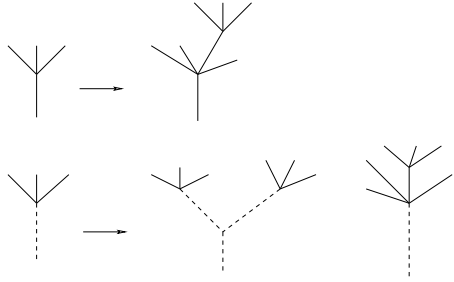}
\caption{The $A_\infty$-relations (top) and $A_{\infty}$-functor relations (below) are sums of compositions as depicted on the right sides of the figure.\label{Fig:Ainftytree}}
\end{figure}
\end{center}

\begin{Remark} A \emph{strictly unital} $A_{\infty}$-category $\scrA$ is one which has unit elements $e_X \in hom^0_{\scrA}(X,X)$ for each object $X$, such that the $e_X$ are closed, $\mu^1_{\scrA}(e_X)=0$, and 
\[
\mu^2_{\scrA}(a,e_X) = a = (-1)^{|a|}\mu^2_{\scrA}(e_X, a), \quad \mu^k_{\scrA}(\ldots, e_X, \ldots) = 0 \ \forall \, k\geq 3.
\]
$\scrA$ is \emph{cohomologically unital} if $H(\scrA)$ is unital, i.e. is a linear category in the usual sense. Any cohomologically unital $A_{\infty}$-category is quasi-isomorphic to a strictly unital one \cite{FCPLT}. 
\end{Remark}

\subsection{$A_{\infty}$-modules and twisted complexes}

An $A_{\infty}$-module $\scrM$ is an element of the category $nu$-$fun(\scrA^{opp},Ch)$, with $Ch$ the $dg$-category of chain complexes.  Concretely, this associates to any $X\in$ Ob$\,\scrA$ a $\bZ$ or $\bZ_2$-graded $\bK$-vector space $\scrM(X)$, and there are maps for $k\geq 1$
\[
\mu_{\scrM}^k: \scrM(X_{k-1}) \otimes hom_{\scrA}(X_{k-2},X_{k-1}) \otimes \cdots \otimes hom_{\scrA}(X_0,X_1) \rightarrow \scrM(X_0)[2-k].
\]
The $A_{\infty}$-functor equations imply that $\mu_{\scrM}^1$ is a differential on the chain complex  $\scrM(X_0)$. For any $Y \in$ Ob$\,\scrA$, there is an associated \emph{Yoneda} module $\scrY$ defined by
\[
\scrY(X) = hom_{\scrA}(X,Y); \quad \mu_{\scrY}^d = \mu_{\scrA}^d. 
\]
The association $Y \mapsto \scrY$ extends to a canonical non-unital $A_{\infty}$-functor $\scrA \rightarrow mod$-$(\scrA)$, the Yoneda embedding.  The category of $A_{\infty}$-modules is naturally a \emph{triangulated} $A_{\infty}$-category:  morphisms have cones.  More precisely, if $c\in hom_{\scrA}(Y_0,Y_1)$ is a degree zero cocycle, $\mu^1_{\scrA}(c)=0$, there is an $A_{\infty}$-module Cone$(c)$ defined by
\begin{equation}\label{eqn:cone}
\mathrm{Cone}(c)(X) = hom_{\scrA}(X,Y_0)[1] \oplus hom_{\scrA}(X,Y_1)
\end{equation}
and with operations $\mu^d_{Cone(c)}((b_0,b_1),a_{d-1},\ldots,a_1)$ given by the pair of terms
\[
\left(\mu_{\scrA}^d(b_0,a_{d-1},\ldots,a_1), \mu_{\scrA}^d(b_1,,a_{d-1},\ldots,a_1)+ \mu_{\scrA}^{d+1}(c,b_0,a_{d-1},\ldots,a_1) \right).
\]
One can generalise the mapping cone construction and consider \emph{twisted complexes} in $\scrA$. First, the additive enlargement $\Sigma\scrA$ has objects formal direct sums
\begin{equation} \label{eq:twisted_complex_vector_space}
X = \oplus_{i\in I} V^i \otimes X^i
\end{equation}
with $\{X^i\} \in$ Ob$\,\scrA$ and $V^i$ finite-dimensional graded $\bK$-vector spaces. A  twisted complex is a pair $(X,\delta_X)$ with $X \in \Sigma\scrA$ and $\delta_X \in hom^1_{\Sigma\scrA}(X,X)$  a matrix of differentials
\begin{equation} \label{eq:twisted_complex_differential}
\delta_X = (\delta_X^{ji});\qquad \delta_X^{ji} = \sum_k \phi^{jik} \otimes x^{jik}
\end{equation}
with $\phi^{jik} \in Hom_{\bK}(V^i,V^j)$, $x^{jik} \in hom_{\scrA}(X^i,X^j)$ and having total degree $|\phi^{jik}| + |x^{jik}| = 1$.  The differential $\delta_X$ should satisfy the two properties
\begin{itemize}
\item $\delta_X$ is strictly lower-triangular with respect to some filtration of $X$;
\item $\sum_{r=1}^{\infty} \mu^r_{\Sigma\scrA}(\delta_X,\ldots,\delta_X)=0$.
\end{itemize}
Twisted complexes form the objects of an $A_{\infty}$-category Tw$(\scrA)$, which has the property that all morphisms can be completed with cones to sit in exact triangles.  Moreover, exact triangles can be manipulated rather effectively, for instance they can be ``rolled up":
\begin{equation} \label{Eqn:RollTriangle}
\xymatrix{ X_0 \ar^{a}[r] & X_1 \ar[r] & \ar@/^1.5pc/[ll]_{[1]} Y}, \ 
\xymatrix{ X_1 \ar^{b}[r] & X_2 \ar[r] & \ar@/^1.5pc/[ll]_{[1]} Z} \ \Rightarrow \ 
\xymatrix{ X_0 \ar^{ba}[r] & X_2 \ar[r] & \ar@/^1.5pc/[ll]_{[1]} \mathrm{Cone}(Z\to Y)}
\end{equation}
which means that one can also manipulate the associated long exact sequences in cohomology. By taking the tensor product of an object $X$ by $\bK$ placed in some non-zero degree, one sees Tw$(\scrA)$ has a \emph{shift functor}, which has the effect of shifting degrees of all morphism groups down by one.  In the $\bZ_2$-graded case, the square of this functor is trivial.

Finally, the \emph{derived category} $D^{\pi}(\scrA) = H(\Tw^{\pi}(\scrA))$ is given by taking the idempotent completion of Tw$(\scrA)$, which incorporates new objects so that any cohomological idempotent morphism in $H(\Tw\, \scrA)$ splits in Tw$^{\pi}(\scrA)$, and then passing to cohomology.  This is also sometimes called the category of \emph{perfect modules}, denoted Perf$(\scrA)$.   Despite its rather involved construction, this is an $A_{\infty}$-category with good properties: being both triangulated and idempotent complete, it can often be reconstructed (up to an unknown quasi-isomorphism) from a finite amount of data, see Section \ref{Sec:Generation}, in a way that neither $\scrA$ nor Tw$(\scrA)$ themselves can be.

\newcommand{\Chord}{{\EuScript X}}
\newcommand{\ro}{{\mathrm{or}}}
\newcommand{\scrR}{\EuScript R}

\subsection{Holomorphic polygons}

Let $X$ be a symplectic manifold; for simplicity we will suppose this satisfies some exactness  or monotonicity, and restrict to exact or monotone Lagrangians, rather than grappling with obstruction theory.  We follow \cite{FCPLT, Sheridan3} in constructing the Fukaya category via consistent choices of Hamiltonian perturbations.

 To start, fix a (usually finite) collection $\scrL$ of Lagrangians which we require to be exact or monotone, oriented, and compact.  The Fukaya category $\scrF(X)$ will be a $\bK$-linear $A_{\infty}$-category with objects the elements of $\scrL$.  If char$(\bK)\neq 2$ then we assume that the elements of $\scrL$ are also equipped with spin structures.  Given a pair  $(L_0, L_1)$ of  Lagrangians in $\scrL$, fix a compactly supported Hamiltonian function 
\begin{equation}
  H_{L_0,L_1} :  [0,1] \times X \rightarrow \bR
\end{equation}
whose time-$1$ flow maps $L_0$ to a Lagrangian $\phi_H^1(L_0)$ transverse to $L_1$. Let $\Chord(L_0,L_1)$ denote the set of intersection points $\phi_H^1(L_0) \pitchfork L_1$, and set 
\begin{equation}
  CF^*(L_0,L_1) \equiv \bigoplus_{x \in \Chord(L_0,L_1)}  \ro_{x}
\end{equation}
where $\ro_x$ is a $1$-dimensional $\bK$-vector space of ``coherent orientations" at $x$, see \cite[Section 11h]{FCPLT}, which is associated to $x$ through the index theory of the $\overline{\partial}$-operator.  We take a time-dependent almost complex structure given by a map $J: [0,1] \rightarrow \scrJ$, which defines an $\bR$-translation invariant structure on the domain $[0,1] \times \bR$. The differential in $ CF^*(L_0,L_1)  $ counts Floer trajectories, i.e. solutions of the equations \eqref{eq:Floer} defined with respect to $J$, which are rigid (virtual dimension zero). We assume that $H_{L_0,L_1}  $ and $J$ are chosen generically so that these moduli spaces are regular, see \cite{FHS}.

To define the $A_{\infty}$-structure on the category, for $k\geq 2$ let $\scrR^{k+1}$ denote the moduli space of discs with $k+1$ punctures on the boundary; these domains are all rigid, and the moduli space of such is the Stasheff associahedron. Let $\Delta$ denote the unit disc in $\bC$. As in \cite[Section 12g]{FCPLT}, we orient $\scrR^{k+1}$ by fixing the positions of $p_0$, $p_1$, and $p_2$, thereby identifying the interior of $  \scrR^{k+1} $ with an open subset of $(\partial \Delta)^{k-2}$, which is canonically oriented.  Having fixed the distinguished puncture $p_0$, the remainder $\{p_1,\ldots, p_k\}$ are ordered counter-clockwise along the boundary. Following \cite[Section 9g]{FCPLT},  choose families of strip-like ends for all punctures. Concretely, we set
\begin{equation} \label{eqn:in-out-striplike-ends}
  Z_- = (-\infty,0] \times [0,1] \textrm{ and } Z_+ = [0,\infty) \times [0,1]
\end{equation}
and choose, for each surface $\Sigma \in \scrR^{k+1}$,  conformal embeddings
\begin{equation}
\epsilon_{0}: Z_- \rightarrow \Sigma, \quad \epsilon_i: Z_+ \rightarrow \Sigma \quad  \textrm{for} \, 1 \leq i \leq k\end{equation}
 which take $\partial Z_{\pm}$ into $\partial \Sigma$, and which converge at the end to the punctures $p_i$. The differences in the choices of ends reflects the fact that $p_0$ will be an output and the other $p_i$'s inputs to the $A_{\infty}$-operations.

Given a sequence  $(L_0, \ldots, L_k)$ of objects of $\scrL$ we next choose inhomogeneous data on the  universal curve over $ \scrR^{k+1} $. Given  $\Sigma \in \scrR^{k+1}  $ and $z \in \Sigma$, this data comprises a map 
\begin{equation}
  K : T_{z} \Sigma \to C^{\infty}_{ct}(X, \bR)
\end{equation} 
subject to the constraint that $\epsilon_i^* K = H_{L_{i-1},L_{i}} dt $. (The Deligne-Mumford compactifications of the spaces $\scrR^{k+1}$ have boundary strata which are products of smaller associahedra, and the inhomogeneous data is constructed inductively over these moduli spaces so as to be compatible with that stratification as well as with the strip-like ends.)  
 $K$ defines a $1$-form on $\Sigma$ with values in the space of vector fields on $X$, namely $Y(\xi) = X_{K(\xi)}$ (where $X_H$ is as usual the Hamiltonian vector field associated to $H$). We obtain a Floer (i.e. perturbed Cauchy-Riemann) equation:
\begin{equation} \label{eq:Floer_equation_0-puncture}
 J \left( du(\xi) - Y(\xi)  \right) = du(j \xi) - Y(j \xi) ,
\end{equation}
on the space of maps $u: \Sigma \to X$ taking the boundary segment  from $p_{i-1}$ to $p_{i}$ on the disc  to $L_i$.  We denote by
\[
  \scrR^{k+1}(X | x_0; x_k ,\cdots, x_1)
\]
the space of finite energy solutions which converge to $x_i$ along the end $\epsilon_i$. Standard regularity results imply that, for generic data, this space is  a smooth manifold of dimension
\begin{equation} \label{Eqn:VirDim}
k - 2 +  \deg(x_0) - \sum_{i=1}^k \deg(x_i).
\end{equation}

The Gromov-Floer construction produces a compactification of $  \scrR^{k+1}(X| x_0; x_k ,\cdots, x_1 )  $ consisting of stable discs. The exactness or monotonicity hypotheses we impose ensure that, if the dimension of $\scrR^{k+1}(X | x_0; x_k ,\cdots, x_1)$ is at most 1, the compactification contains no sphere or disc bubbles, and is again a smooth closed manifold of the correct dimension.  The moduli space is oriented relative to the orientation lines by the choice of orientation on the associahedron.  In particular, whenever \eqref{Eqn:VirDim} vanishes, the signed count of elements of  $  \scrR^{k+1}(X|x_0; x_k, \ldots, x_1) $ defines a map:
\begin{equation} \label{eqn:summand}
   \ro_{x_k} \otimes  \cdots \otimes \ro_{x_1} \to \ro_{x_0}.  
\end{equation}
By definition, \eqref{eqn:summand} defines the $\ro_{x_0}$-component of the $A_{\infty}$-operation $\mu_{\scrF(X)}^k$ restricted to the given inputs:
\begin{equation}
   \ro_{x_k} \otimes  \cdots \otimes \ro_{x_1}  \subset CF^*(L_{k-1}, L_k) \otimes \cdots \otimes CF^*(L_{0}, L_1)  \to CF^*(L_{0}, L_k).
\end{equation}
The $A_{\infty}$-relations are algebraic shadows of the structure of the Deligne-Mumford compactification of the moduli space of  stable disks with boundary marked points $\overline{\scrR}^{k+1}$: the terms in the quadratic associativity equations correspond to codimension one facets in the boundary locus of singular stable disks, cf. Figure \ref{Fig:DMboundary} (and compare to Figure \ref{Fig:BrokenTriangles}).  The next compactification $\overline{\scrR}^{5}$ is a pentagon, and the five facets are the terms in the first quadratic equation involving $\mu^4$, etc. 

\begin{center}
\begin{figure}[ht]
\includegraphics[scale=0.7]{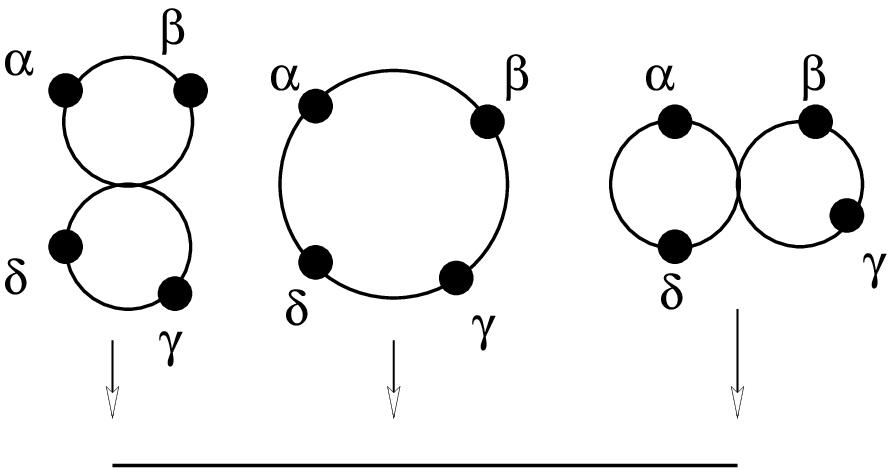} \qquad \includegraphics[scale=0.26]{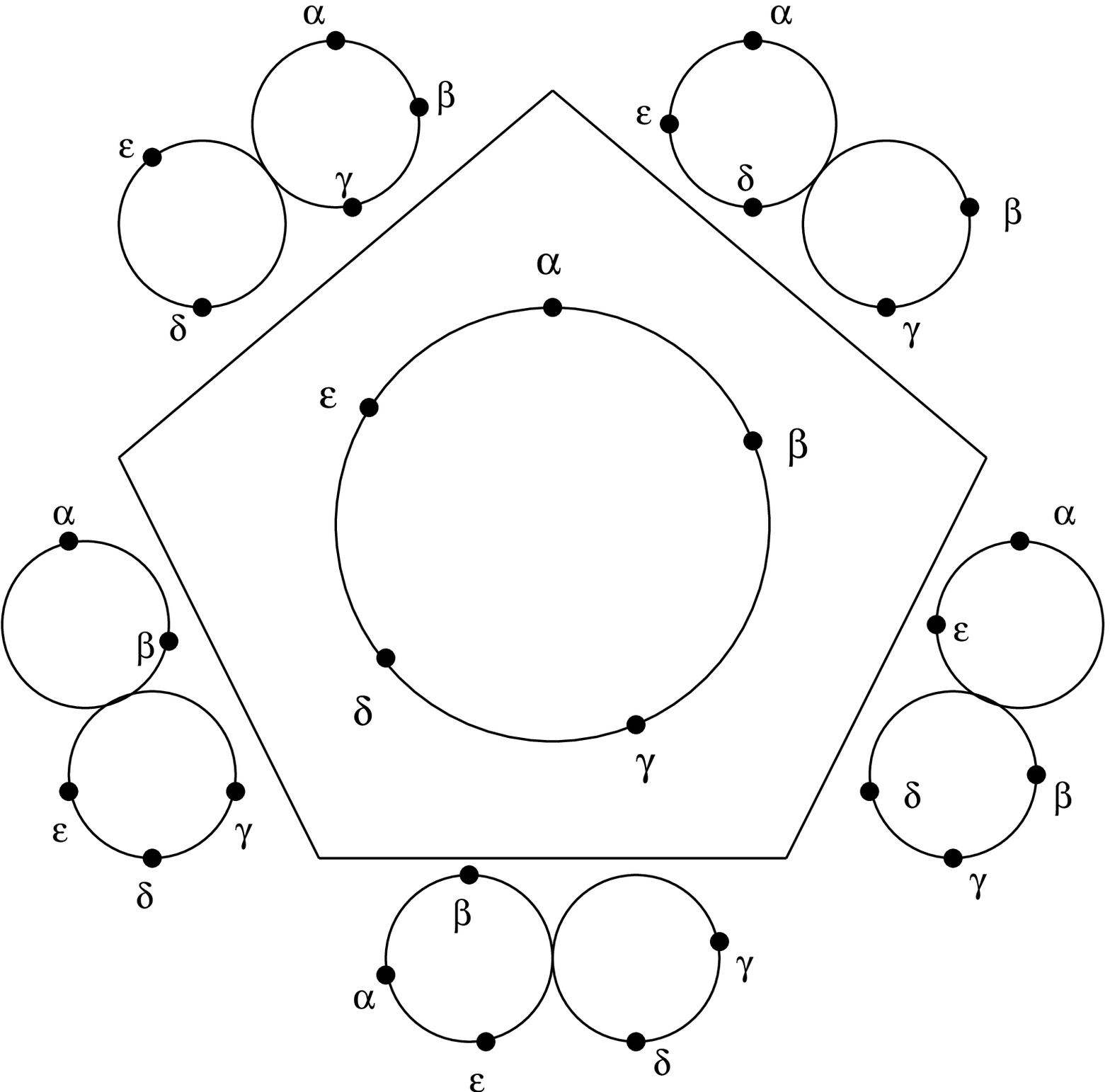}
\caption{The compactifications $\overline{\scrR}^4(\delta;\alpha,\beta,\gamma)$ and $\overline{\scrR}^5(\epsilon; \alpha,\beta,\gamma,\delta)$.\label{Fig:DMboundary}}
\end{figure}
\end{center}

\begin{Example} \label{Ex:CoveringSpace} The Fukaya category is built out of holomorphic discs, which lift  to covering spaces. 
Let $\pi:\tilde{M} \rightarrow M$ be a finite Galois cover of Stein manifolds, and let $\scrF(\cdot)$ denote the exact Fukaya category. There is a functor
$\pi^*: \scrF(M) \longrightarrow \scrF(\tilde{M})$ 
which takes $L \in \scrF(M)$ to the total preimage $\pi^{-1}(L) \subset \tilde{M}$, and such that the map on morphisms $ HF^{*}(L,L) \to HF^{*}(\pi^{-1}(L), \pi^{-1}(L) )$ 
agrees with the classical pullback on cohomology (since the fixed locus of the group action is presumed trivial, equivariant transversality is not problematic).  Deck transformations of $\pi$ act by autoequivalences of $\scrF(\tilde{M})$.
\end{Example}

Fukaya categories are cohomologically unital because $HF^*(L,L)$ is a unital ring.  However, they are not ``geometrically" strictly unital, i.e. there are not typically strict units without applying some abstract quasi-isomorphism.   

In a few cases, one can regard the passage from twisted complexes $\Tw\,\scrF(X)$ to their idempotent completion $\Tw^{\pi}\scrF(X)$ as follows: iterated Lagrange  surgeries (which we can hope to interpret as mapping cones) give rise to Lagrangian submanifolds which may be disconnected, and one is including the individual components in taking split-closure. In general no such geometric interpretation is available: for instance, for the diagonal $S^2 \subset \bP^1\times(\bP^1)^-$ the Floer cohomology $\cong QH^*(\bP^1)$ has an idempotent, and the summands of the object have homology classes $(1\pm[S^2])/2$, which do not live in pure degree.

\subsection{Twists}
A particularly important class of mapping cones are those arising from twist functors.  Given $Y \in$ Ob$\,\scrA$ and an $\scrA$-module $\scrM$, we define the \emph{twist} $\scrT_Y\scrM$ to be the module 
\[
\scrT_Y\scrM(X) = \scrM(Y)\otimes hom_{\scrA}(X,Y)[1] \oplus \scrM(X)
\]
which is the cone over the canonical evaluation morphism
\begin{equation} \label{Eqn:Twist}
\scrM(Y) \otimes \scrY \rightarrow \scrM,
\end{equation}
where $\scrY$ denotes the Yoneda image of $Y$.  For two objects $Y_0, Y_1 \in$ Ob$\, \scrA$, the essential feature of the twist is that it gives rise to a canonical exact triangle in $H(\scrA)$
\[
\cdots \rightarrow Hom_{H(\scrA)}(Y_0,Y_1) \otimes Y_0 \rightarrow Y_1 \rightarrow T_{Y_0}(Y_1) \stackrel{[1]}{\rightarrow} \cdots
\]
where $T_{Y_0}(Y_1)$ is any object whose Yoneda image is $\scrT_{Y_0}(\scrY_1)$; if $H(\scrA)$ has finite-dimensional morphism spaces, such an object necessarily exists, cf. \cite[p. 64]{FCPLT}. 

 For Fukaya categories, the twist functor $\scrT_L \in nu$-$fun(\scrF(X),\scrF(X))$ plays a special role when $L$ is \emph{spherical}, meaning that $\Hom_{\mathrm{Don}(X)}(L,L) \cong H^*(S^n)$.  Suppose $L\subset X$ is a Lagrangian sphere;  the \emph{geometric twist} $\tau_L$ is the autoequivalence of $\scrF(X)$ defined by the positive Dehn twist in $L$, cf. Section \ref{Sec:Appl1}. 
 On the other hand, in the monotone or exact case, such a sphere gives rise to a well-defined (though in the monotone case not \emph{a priori} non-zero) object of the Fukaya category $\scrF(X)$, hence an algebraic twist functor in the sense described above. 

\begin{Theorem}[Seidel] \label{Prop:twists}
If $L\subset X$ is a Lagrangian sphere, equipped with the non-trivial \emph{Spin} structure if $L \cong S^1$, then $\tau_L$ and $\scrT_L$ are quasi-isomorphic in $nu$-$fun(\scrF(X),\scrF(X))$.
\end{Theorem}

This is \cite[Corollary 17.17]{FCPLT} for exact symplectic manifolds, and the argument in the monotone case is similar.    This gives a basic link between algebra and geometry, and strengthens the exact triangle in the Donaldson category which we introduced previously in Section \ref{Subsec:ExactTriDon}.

\subsection{Classification strategies} One strategy for understanding the Fukaya category of a given symplectic manifold $(X,\omega)$ has three basic steps: identify a collection of ``favourite" Lagrangian submanifolds $\{L_i\}$; compute the $A_{\infty}$-structure on the finite-dimensional algebra $\oplus_{i,j} HF^*(L_i, L_j)$; and then prove the $\{L_i\}$ (split-)generate $\scrF(X)$.   The central step is often  roundabout,  to minimise the need for  explicit solutions to the Cauchy-Riemann equations.

\subsubsection{Hochschild cohomology}
The \emph{Hochschild cohomology} of an $A_{\infty}$-category is defined by the following bar complex  $CC^*(\scrA,\scrA)$.  A degree $r$ cochain is a sequence $(h^d)_{d\geq 0}$ of collections of linear maps
\[
h^d_{(X_1,\ldots,X_{d+1})}: \bigotimes_{i=d}^1 hom_{\scrA}(X_i,X_{i+1}) \rightarrow hom_{\scrA}(X_1,X_{d+1})[r-d]
\]
for each $(X_1,\ldots,X_{d+1})\in \Ob(\scrA)^{d+1}$.   The differential is defined by the sum over  concatenations
\begin{equation} \label{Eqn:Hochschild}
\begin{aligned}
(\partial h)^d & (a_d,\ldots, a_1) = \\
& \sum_{i+j<d+1} (-1)^{(r+1)\maltese_i} \mu_{\scrA}^{d+1-j}(a_d,\ldots,a_{i+j+1},h^j(a_{i+j},\ldots,a_{i+1}),a_{i},\ldots,a_1) \\
+ &  \sum_{i+j\leq d+1} (-1)^{\maltese_{i} +r +1}  h^{d+1-j} (a_d,\ldots,a_{i+j+1},\mu_{\scrA}^j(a_{i+j},\ldots,a_{i+1}),a_{i},\ldots,a_1).
\end{aligned}
\end{equation}
More succinctly, in fact  $HH^*(\scrA,\scrA) = H(hom_{fun(\scrA,\scrA)}(\id,\id))$. Hochschild cohomology of an $A_{\infty}$-category over a field is invariant under both taking twisted complexes and idempotent completion \cite[Section 1c]{Seidel:Flux}. 

Let $A$ be a $\bZ_2$-graded algebra over a field $\bK$ which has characteristic not equal to 2.  We will be interested in $\bZ_2$-graded $A_{\infty}$-structures on $A$ which have trivial differential. Such a structure comprises a collection of maps 
\[
\{ \alpha^i: A^{\otimes i} \rightarrow A \}_{i\geq 2}
\]
of parity $i$.  We introduce the mod 2 graded Hochschild cochain complex
\[
CC^{\bullet+1}(A,A) \ = \ \prod_{i \geq 2} Hom^{(\bullet+i)}(A^{\otimes i},A)
\]
where $Hom^{(\bullet +i)}$ means we consider homomorphisms of parity $\bullet+i$. This carries the usual differential $d_{CC^{\bullet}}$ and the Gerstenhaber bracket $[\cdot, \cdot]$.   $A_{\infty}$-structures $(\alpha^i)$ on $A$ which extend its given product correspond to elements of $CC^{1}(A,A)$ satisfying the Maurer-Cartan equation
\begin{equation} \label{Eqn:Maurer-Cartan}
d_{CC^{\bullet}} (\alpha) + \frac{1}{2} [\alpha, \alpha] = 0\qquad \textrm{with} \quad \alpha^1 = \alpha^2 = 0.
\end{equation}
Elements $(g^i)_{i\geq 1}$ of $CC^0(A,A)$, comprising maps $A^{\otimes i} \rightarrow A$ of parity $i-1$,   act by gauge transformations on the set of Maurer-Cartan solutions $(\alpha^i)_{i\geq 2}$, at least when suitable convergence conditions apply. Thus, the set of $A_{\infty}$-structures on a given algebra $A$ is bound up with the Hochschild cochain complex of $A$. To connect the two discussions, if $\scrA$ denotes the $A_{\infty}$-algebra defined by the Maurer-Cartan element $\alpha \in CC^1(A,A)$, with $H(\scrA) = A$, then $HH^*(\scrA,\scrA)$ is the cohomology of the complex $CC^{\bullet}(A,A)$ with respect to the twisted differential $d_{CC^{\bullet}} + [\cdot, \alpha]$. 

The closed-open map $\mathcal{CO}^0: QH^*(X) \rightarrow HF^*(L,L)$ from \eqref{Eqn:ClosedOpen} extends, by considering discs with an interior input and several boundary inputs as well as a single boundary output, to a map
\begin{equation} \label{Eqn:COHH}
\mathcal{CO}: QH^*(X) \longrightarrow HH^*(\scrF(X), \scrF(X))
\end{equation}
which we still call the closed-open string map; composing this with the truncation map which projects to the length zero part of the Hochschild complex recovers the map \eqref{Eqn:ClosedOpen} introduced previously.  

\begin{Remark}\label{Rem:StarC1}
In Section \ref{Subsec:ModuleStructures} we remarked that Auroux, Kontsevich and Seidel proved that for monotone symplectic manifolds, the possible summands of the Fukaya category $\scrF(X)$ are indexed by the eigenvalues of $\ast c_1(X)$.  Formally, their argument starts from the observation that the composition
\[
\mathcal{CO}^0: \ QH^*(X) \longrightarrow HH^*(\scrF(X), \scrF(X)) \longrightarrow HF^*(L,L)
\]
of \eqref{Eqn:COHH} with truncation to the length zero part of the Hochschild complex satisfies  
\begin{equation} \label{Eqn:2c1degrees}
\mathcal{CO}^0(2c_1(X)) = 2 \frak{m}_0(L) \in HF^*(L,L).
\end{equation}  
Recalling that $\mu_L \mapsto 2c_1(X)$ under $H^2(X,L) \rightarrow H^2(X)$ and taking a representative $D$ for the Maslov cycle  $\mu_L \in H^2(X,L) = H^2_{ct}(X\backslash L)  \cong H_{2n-2}(X\backslash L)$ disjoint from $L$, to eliminate contributions from constant discs,  both sides of \eqref{Eqn:2c1degrees} count essentially the same thing: a Maslov index two disc with boundary on $L$ and with an interior marked point constrained to lie on $D$.  For index reasons,  $\frak{m}_0(L) \in HF^*(L,L)$ is a multiple of the unit $e_L$ (tautologically, the zero multiple if $L$ bounds no Maslov index two disc). If $\frak{m}_0(L) = \lambda e_L$ then 
\[
QH^*(X) \rightarrow HF^*(L,L) \quad \mathrm{takes} \ 2(c_1-\lambda 1) \mapsto 0
\]
and hence $c_1-\lambda 1$ is not invertible, in characteristic $\neq 2$.  \end{Remark}

\subsubsection{Finite determinacy\label{Subsec:FiniteDeterminacy}}
Just as the product on Floer cohomology has an underlying $A_{\infty}$-structure at chain level, the Lie bracket on Hochschild cohomology has an underlying $L_{\infty}$-structure at chain level.  Kontsevich's formality theorem \cite{Kontsevich:DefQuant} asserts that if $A$ is a smooth commutative algebra,  under suitable filtered nilpotence conditions there is a quasi-isomorphism of $L_{\infty}$-algebras $CC^*(A,A) \rightarrow HH^*(A,A)$ which induces a bijection on the set of Maurer-Cartan solutions. This enables one to transfer existence and classification results for $A_{\infty}$-structures on $A$ into questions about Hochschild cohomology and the action thereon of gauge transformations. 

The conclusion is well illustrated by taking $A = \Lambda(V)$ to be a finite-dimensional exterior algebra. The Hochschild-Kostant-Rosenberg theorem identifies Hochschild cohomology with the Lie algebra of polyvector fields on $V$, namely 
\[
Sym^{\bullet}(V^{\vee}) \otimes \Lambda(V) \ = \ \bC[\![V]\!] \otimes \Lambda(V).
\]
The Lie algebra structure comes from the Schouten bracket
\begin{equation}
\begin{aligned} {}
& [f \,\xi_{i_1} \wedge \cdots \wedge \xi_{i_k}, g \,\xi_{j_1} \wedge \cdots \xi_{j_l}] = \\
& \textstyle \qquad \sum_q (-1)^{k-q-1} f \, (\partial_{i_q}g) \,
\xi_{i_1}  \wedge \cdots \wedge \widehat{\xi_{i_q}} \wedge \cdots
\wedge \xi_{i_k} \wedge \xi_{j_1} \wedge \cdots \wedge \xi_{j_l} + \\ & \textstyle \qquad
 \sum_q (-1)^{l-q+ (k-1)(l-1)} g \,(\partial_{j_q}f)\, \xi_{j_1} \wedge \cdots \wedge \widehat{\xi_{j_q}} \wedge \cdots \wedge \xi_{j_l} \wedge \xi_{i_1} \wedge
 \cdots \wedge \xi_{i_k}.
\end{aligned}
\end{equation}
Since the Schouten bracket vanishes on functions, Kontsevich's theorem implies that any formal function $W\in\bC[\![V]\!]$ defines a solution of Maurer-Cartan, hence an $A_{\infty}$-structure on $\Lambda(V)$; and that two functions which differ by pullback by a formal diffeomorphism define quasi-isomorphic $A_{\infty}$-structures.  

Suppose now $W$ has an isolated critical point at $0\in V$. Then classical finite determinacy for singularities implies that $W$ is equivalent under formal change of variables to a polynomial. In this sense, the $A_{\infty}$-structure defined by $W$ is completely determined by a finite amount of data (up to quasi-isomorphism), hence is in principle computable, even though \emph{a priori} it may admit infinitely many non-zero products.   See Section \ref{Subsec:genus2} for an outline example.

\subsection{Generation\label{Sec:Generation}}
It is important to have a criterion for a finite collection of Lagrangians to generate the Fukaya category.  Explicitly, given $L_1,\ldots,L_k \subset X$, there is an $A_{\infty}$-algebra 
$\scrA \ = \ \oplus_{i,j} HF^*(L_i,L_j)$.
One says that the $L_i$ split-generate $\scrF(X)$ if every object of $\scrF(X)$ is quasi-isomorphic to a summand of some twisted complex on the $\{L_i\}$, which implies the functor
\[
\scrF(X) \longrightarrow  \mathrm{mod}_{\!A_{\infty}}(\scrA), \qquad K \mapsto \oplus_i \, HF^*(K,L_i)
\]
is a fully faithful embedding. We will mention two sufficient criteria; the latter is stronger, but the former is applicable in some interesting situations.

\subsubsection{Positive relations} Suppose first $\{L_1,\ldots, L_k\} \subset X$ are all Lagrangian spheres, and furthermore that they satisfy a \emph{positive relation}, meaning that some word $w \in \langle \tau_{L_1}, \ldots, \tau_{L_k}\rangle$   in positive Dehn twists $\tau_{L_i}$ is Hamiltonian isotopic to the identity. The data of $w$, together with the isotopy, define a symplectic Lefschetz fibration $\scrX \rightarrow \bP^1$ with generic fibre symplectomorphic to $X$, whose monodromy over $\bC \subset \bP^1$ is given by $w$, and which is completed to a fibration over $\infty \in \bP^1$ via the given Hamiltonian isotopy. The point of insisting that the isotopy be Hamiltonian (and not just symplectic) is that the total space of $\scrX$ then carries a symplectic structure which extends the fibrewise structure coming from $X$. In particular, there is a good theory of holomorphic curves in $\scrX$.
In stating the following Proposition, we assume that $X$ is monotone and work over $\bC$, so as in Remark \ref{Rem:StarC1},  $\scrF(X) = \oplus_{\lambda} \scrF(X;\lambda)$ with $\lambda \in \mathrm{Spec}(\ast c_1(X))$.

\begin{Proposition} \label{Prop:SpheresGenerate}
Let $\mathcal{C}(w) \in QH^*(X)$ denote the cycle obtained by evaluating holomorphic sections of $\scrX\rightarrow \bP^1$ at a base-point. Suppose the $L_i$ all belong to $\scrF(X; \lambda)$, and that quantum cup-product by $\mathcal{C}(w)$ on $QH^*(X;\lambda) $ is nilpotent. Then the $\{L_i\}$ split-generate $\scrF(X;\lambda)$. In particular, if $\lambda =0$ and $\mathcal{C}(w)$ is a multiple of $c_1(X)$, the $L_j$ split-generate $\scrF(X;0)$.
\end{Proposition}
  
\begin{proof}
By the correspondence between algebraic and geometric Dehn twists, Proposition \ref{Prop:twists}, there are exact triangles
\[
\cdots \rightarrow  HF(L,K)\otimes L \rightarrow K \rightarrow \tau_L(K) \stackrel{[1]}{\longrightarrow} \cdots
\]
for any Lagrangian $K \subset X$ lying in the $\lambda$-summand of the category.  Concatenating the triangles for the $\tau_{L_j}$ occuring in $w$, as in \eqref{Eqn:RollTriangle}, defines a natural map
\begin{equation}\label{eqn:concatenate}
K \rightarrow \prod_{i_j} \tau_{L_{i_j}} K \cong K
\end{equation}
defined by an element of $HF^*(K,K)$.  From the original construction of the long exact sequence in Floer cohomology \cite{Seidel:LES} and Section \ref{Subsec:ExactTriDon},  the natural map $K \rightarrow \tau_L(K)$ arises from counting sections of a Lefschetz fibration over an annulus with a single critical point. The gluing theorem \cite[Proposition 2.22]{Seidel:LES} implies that the concatenation of such maps counts holomorphic sections of the Lefschetz fibration given by sewing several such annuli together.   Under the natural restriction map $r: H^{*}(M) \rightarrow HF^{*}(K,K)$, Equation \ref{eqn:concatenate} is given by multiplication by the image of the cycle class $r(\mathcal{C}(w)) \in HF^{*}(K,K)$.  If $\mathcal{C}(w)$ vanishes, one map in the iterated exact triangle vanishes, and  $K$ is a summand in an iterated cone amongst the  $\{L_i\}$.  If $\ast \mathcal{C}(w)$ is nilpotent, we can run the same argument after replacing $w$ by $w^N$ for $N \gg 0$.
\end{proof}

\subsubsection{Abouzaid's criterion}
The obvious drawback of Proposition \ref{Prop:SpheresGenerate} is that the $L_i$ need to be spheres.  A better criterion was given by Abouzaid \cite{Abouzaid:generation}, in the exact case, and by Abouzaid-Fukaya-Oh-Ohta-Ono (in preparation) in general; explanations of the latter for monotone manifolds are given in \cite{RitterSmith, Sheridan3}.

\begin{Theorem} [Abouzaid-Fukaya-Oh-Ohta-Ono] \label{Thm:AFOOOgenerate}\label{Thm:Generation}
If the map $\mathcal{CO}: QH^*(X) \rightarrow HH^*(\scrA,\scrA)$ is an isomorphism, then the $L_i$ split-generate $\scrF(X)$.
\end{Theorem}

It is instructive to consider a toy-model case of the situation considered in Theorem \ref{Thm:Generation}, which works at the level of the length zero part of Hochschild cohomology, where the closed-open map of \eqref{Eqn:COHH} reduces to the simpler \eqref{Eqn:ClosedOpen} considered before. Suppose there is a Lagrangian submanifold $L\in \scrF(X;\lambda)$ with the property that 
\[
\mathcal{CO}^0: QH^*(X;\lambda) \rightarrow HF^*(L,L)
\]
is an isomorphism. We will show that in that case $L$ split-generates $\scrF(X;\lambda)$.  Indeed, for any $K\in \scrF(X;\lambda)$, it will suffice to show the map
\begin{equation} \label{Eqn:HitUnit}
HF^*(L,K) \otimes HF^*(K,L) \longrightarrow HF^*(K,K)
\end{equation}
hits the unit $1_K$.  To check that, write  $1_K = \sum_{j=1}^N p_j \cdot q_j$ in \eqref{Eqn:HitUnit}. There are maps 
\[
\alpha = \sum_{j=1}^N \mu^2(p_j, \cdot) \quad \mathrm{and} \quad \beta = (\mu^2(q_1,\cdot),\ldots, \mu^2(q_N,\cdot))
\]
 which fit into a diagram
\[
\xymatrix{
\bigoplus_{j=1}^N \, HF^*(L,L)\ar@<+2pt>@{->}^-{\alpha}[rr] \ar@<-2pt>@{<-}_-{\beta}[rr] && HF^*(L,K)
& \alpha \circ \beta  =1_K} 
\]
which shows that $\beta \circ \alpha$ is idempotent, and has image a copy of $HF^*(L,K) \subset \oplus_j HF^*(L,L)$. Given this, $K \in \Tw^{\pi}\langle L\rangle $ is a direct summand of copies of $L$ (quasi-isomorphic to the image of the idempotent $\beta\circ\alpha$, which is introduced formally as an object in taking the split-closure). 

\begin{center}
\begin{figure}[ht]
\includegraphics[scale=0.5]{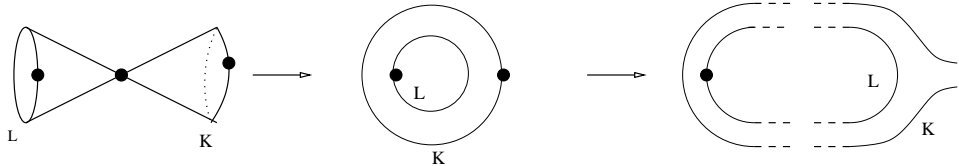}
\caption{A variant of the Cardy relation\label{Fig:Cardy2}}
\end{figure}
\end{center}

To see \eqref{Eqn:HitUnit} hits the unit, we study a 1-parameter family of annuli, see Figure \ref{Fig:Cardy2} (such arguments go by the name of Cardy relations, and go back in this context to \cite{BiranCornea}).   On the left side of that picture, we fix marked points on the $L$ and $K$ boundary components.  By considering that nodal annulus as a disc in $X\times X$ with an  incidence condition at the origin, the picture represents the image of the diagonal
\begin{equation} \label{Eqn:Nodal}
\Delta \subset QH^*(X) \otimes QH^*(X) \longrightarrow HF^*(L,L) \otimes HF^*(K,K).
\end{equation}
If tr$: HF^*(L,L) \cong QH^*(X;\lambda) \rightarrow \bC$ is the restriction of the usual trace on $QH^*(X)$ given by pairing with the fundamental class, we therefore get a class $(\tr \otimes \id)(\Delta) \in HF^*(K,K)$. On the other hand,  the annulus in the middle of Figure \ref{Fig:Cardy2} (with its two boundary marked points) can be broken at the other end of a one-dimensional degeneration into a strip with a boundary marked point and a three punctured disc.  Using the trace on $HF^*(L,L)$ to cap off the boundary marked point  in this setting  amounts to considering the composite 
\begin{equation} \label{Eqn:Smoothed}
\bK \rightarrow HF^*(L,K) \otimes HF^*(L,L) \otimes HF^*(K,L) \stackrel{\mathrm{tr}}{\longrightarrow} HF^*(L,K) \otimes HF^*(K,L) \stackrel{\mu^2}{\longrightarrow}HF^*(K,K). 
\end{equation}
Via this 1-parameter family of annuli, one sees that $(\tr\otimes \id)(\Delta)$ lies in the image of \eqref{Eqn:Smoothed}, and hence in the image of multiplication \eqref{Eqn:HitUnit}.  Since the diagonal is given by $[X]\otimes \id + $(other terms), and $QH^*(X) \rightarrow HF^*(K,K)$ is unital and factors through $QH^*(X;\lambda)$, one sees by comparing \eqref{Eqn:Nodal} and \eqref{Eqn:Smoothed} that the image of \eqref{Eqn:HitUnit} contains the unit, as required.

\subsection{Applications, III} We contine our running theme of illustrating the algebraic machinery in some typical situations; see \cite{FSS,Nadler} respectively \cite{AbouzaidSmith:plumbing}.

\subsubsection{Nearby Lagrangians\label{Sec:NearbyLag}}  It is a non-trivial theorem that every exact Lagrangian submanifold of the cotangent bundle is isomorphic in the Fukaya category to the zero-section, so the compact Fukaya category only contains one object.  The proof of this result goes via a description of a more interesting category which incorporates non-compact Lagrangian submanifolds, namely the cotangent fibres.  The essential point is that these fibres generate (suitable versions of) the Fukaya category; this fact was already prefigured in \eqref{Eqn:FibresShouldGenerate}.

There are several different approaches at a technical level, but any of them lead to the following. Let $L\subset T^*Q$ be an exact Lagrangian with vanishing Maslov class. The Floer cohomology groups $HF^*(T^*_xQ,L)$  form a flat bundle of $\Z$-graded vector spaces over $Q$, which we denote by $E_L$.  One associates to $L\subset T^*Q$ the underlying chain-level object, which is a $\bZ$-graded $dg$-sheaf $\scrE_L$  over $Q$ with fibre $CF^*(L, T_x^*Q)$, and argues that this association fully faithfully embeds the Fukaya category $\scrF(T^*Q)$ into a category of graded $dg$-sheaves on $Q$.  There is the usual local-to-global spectral sequence for computing Ext's of sheaves
\[
E_1 = H^*(Q; \mathcal{E}xt(\scrE_L,\scrE_L)) \ \Longrightarrow \ \mathrm{Ext}^*(H(\scrE_L), H(\scrE_L))
\]
and the target group can be identified with Floer cohomology $HF^*(L,L)$ by fullness and faithfulness of the embedding of $\scrF(T^*Q)$. The $E_2$-page of the spectral sequence then gives 
\begin{equation} \label{eq:ss}
E_2^{rs} = H^r(Q;\End^s(E_L)) \Rightarrow H^*(L;\bK)
\end{equation}
with $\End^*(E_L)$ the graded endomorphism 
bundle. Because $\pi_1(Q)=0$, $E_L$ is  a trivial bundle, so the $E_2$ page is a ``box'' $H^*(Q) \otimes \End(HF^*(T^*_x,L))$. The $E_2$ level differential goes from $(r,s)$ to $(r+2,s-1)$, and similarly for the higher pages. Hence, the bottom left and top right corners of the box necessarily survive to $E_\infty$. Since $E_{\infty}$ is concentrated in degrees $0\leq \ast \leq n=\mathrm{dim}(Q)$,  it follows that $HF^*(T^*_x,L) \iso \bK$ must be one-dimensional, and projection $L \rightarrow Q$ has degree $1$. Given that, the spectral sequence degenerates, yielding $H^*(L;\bK) \iso H^*(Q;\bK)$.  There is also a spectral sequence
\begin{equation} \label{eq:ss2}
E_1 = H^*(Q; \mathcal{E}xt(\scrE_Q,\scrE_L)) \ \Longrightarrow \ HF^*(Q,L).
\end{equation}
A similar analysis shows the edge homomorphism is an isomorphism in top degree $n$; that edge homomorphism has an interpretation in terms of a Floer product, and one concludes that 
\[
HF^n(T_x^*Q, L) \otimes HF^0(Q, T_x^*Q) \longrightarrow HF^n(Q,L) 
\] 
is an isomorphism,  for any $x\in Q$.  It then follows that there is an element $\alpha \in HF^0(Q,L)$ with the property that
\[
\mu^2(\alpha,\cdot): HF^*(T_x^*,Q) \rightarrow HF^*(T_x^*,L)
\]
is an isomorphism, but since fibres generate the category, that is exactly saying that $\alpha$ defines an isomorphism between $Q$ and $L$ in $\scrF(T^*Q)$.

\subsubsection{Mapping class groups}\label{Sec:MCGIII}
Part of  Theorem \ref{Prop:twists} is the statement that the algebraic twist functor $T_L$, which exists for any object $L \in \scrF(X)$, is invertible when $L$ is a sphere.  That argument can be made purely algebraic, and would apply equally well when $L=\Sigma$ is any integral homology sphere.  Let $M_{\Sigma} = T^*\Sigma\#T^*S^n$ be the manifold given by plumbing the given cotangent bundles, i.e. by capping off one fibre of $T^*\Sigma$ to a closed Lagrangian sphere (this surgery can be performed in the Stein category). The $\bZ_2$-graded exact Fukaya category $\scrF(M_{\Sigma})$ admits the  endofunctor $T_{\Sigma}$,  which is an equivalence since $\Sigma$ is a homology sphere. That equivalence is easily seen to be infinite order.  Suppose $\pi_1(\Sigma)$ is finite but not trivial.  We claim that $T_{\Sigma}$ is \emph{not} induced by any compactly supported symplectomorphism, i.e. \emph{there is no geometric Dehn twist}. (We introduce $M_\Sigma$ to give an argument wholly in terms of closed Lagrangians, rather than with the non-compact cotangent fibres.)

 Let $\iota: \tilde{\Sigma} \rightarrow \Sigma$ denote the universal covering, and $\pi: \tilde{M}_{\Sigma} \rightarrow M_{\Sigma}$  the induced covering. Fix a coefficient field $\bK$ of characteristic dividing the index of  $\iota$.   Suppose for contradiction that the algebraic twist $T_{\Sigma}$ is  induced by a compactly supported symplectomorphism $\tau$.  Considering $\tau \circ \tau$, or its inverse, and recalling the definition of the twist functor in Equation \ref{Eqn:Twist}, we infer that there is an exact Lagrangian submanifold $L \subset M_{\Sigma}$ representing the twisted complex 
\begin{equation} \label{Eqn:ImpossibleComplex}
\Sigma \leftarrow \Sigma \leftarrow S^n
\end{equation}
with the leftmost arrow given by multiplication by the fundamental class $[\Sigma]$, and where $S^n$ is the other core component.  Now recall Example \ref{Ex:CoveringSpace}. The inverse image of \eqref{Eqn:ImpossibleComplex} in $\scrF(\tilde{M}_{\Sigma})$ is represented by $\tilde{\Sigma} \leftarrow \tilde{\Sigma} \leftarrow \pi^{-1}(S^n)$, 
in which \emph{the left differential vanishes}: it is obtained by pullback from the fundamental class of $\Sigma$, so is zero in the coefficient field $\bK$ by our choice of characteristic of $\bK$.  

Being the preimage of a connected Lagrangian submanifold $L$ under a covering, all the components of $\pi^{-1}(L)$ are diffeomorphic and moreover related by deck transformations and hence autoequivalences of  $\scrF(\tilde{M}_{\Sigma})$. The complex $\Sigma \leftarrow S^n$ is the image of $\Sigma$ under the Dehn twist in $S^n$.  Therefore, both $\tilde{\Sigma}$ and $\tilde{\Sigma} \leftarrow \pi^{-1}(S^n)$ are represented by connected Lagrangians, hence are indecomposable objects (we are in a $\bZ$-graded situation, and $HF^0$ has rank $1$ for a connected Lagrangian). They are not in the same orbit under deck transformations, again by Example \ref{Ex:CoveringSpace}: on taking Floer cohomology with a cotangent fibre to one of the components of $  \pi^{-1}(S^n)$,  the Floer group has rank $0$ in one case, and $1$ in the other. The previous paragraph showed that, up to shifts,
\begin{equation} \label{Eqn:SplitPreimage}
\pi^{-1}(L) \cong \tilde{\Sigma} \oplus \left( \tilde{\Sigma} \leftarrow \pi^{-1}(S^n) \right)
\end{equation}
in the category $\scrF(\tilde{M}_{\Sigma})$.  We have now shown that the RHS of \eqref{Eqn:SplitPreimage} is not quasi-isomorphic to a direct sum of indecomposables lying in a single orbit of the covering group, which is a  contradiction.

\begin{Notes} For background on $A_{\infty}$-algebra and the Fukaya category, see \cite{FCPLT, FO3}.  For Hochschild cohomology and deformation theory arguments, see \cite{Seidel:HMSquartic, Seidel:HMSgenus2}.
\end{Notes}

\section{Examples}

We describe the Fukaya categories of various manifolds ``explicitly". The examples are real surfaces (with Lagrangian circles which bound no discs), Calabi-Yau surfaces (governed by Example \ref{Ex:CY}) or monotone varieties, so these categories can be defined by elementary means.

\subsection{The 2-torus} Take $\bK = \Lambda$. An application of Proposition \ref{Prop:SpheresGenerate} shows that the longitude $l$ and meridian $m$ split-generate $\scrF(T^2)$.  Indeed, $(\tau_l \tau_m)^6 = \id$, the corresponding Lefschetz fibration $\scrX \rightarrow \bP^1$ is a rational elliptic surface, and the cycle class $\mathcal{C}(w)$ counts the finitely many $(-1)$-curve sections, hence evaluates to a non-zero multiple of the top class in $H^2(T^2)$.  The circles $l$, $m$ bound no holomorphic discs so have self-Floer cohomology equal to $H^*(S^1)$.  The two circles meet transversely at a single point, which defines generators $[p]\in HF^0(l,m)$ and $[q]\in HF^1(m,l)$. Here we choose the grading for $p$ and that for $q$ is forced by Poincar\'e duality, which also  implies that the products $\mu^2(p,q)$ and $\mu^2(q,p)$ are non-trivial.  For grading reasons these are the only non-trivial products except those forced by unitality.  
It follows that the six-dimensional algebra $A = HF(l \oplus m, l \oplus m)$ can be described in terms of the quotient of the path algebra $\Pi\,Q$ of the quiver $Q$ below
\begin{equation*}
\xymatrix{
 l \ \ar@/^/^{[p]}[r] & \ m\ar@/^/^{[q]}[l]
}
\end{equation*}
by the ideal generated by the compositions $a\rightarrow b\rightarrow a \rightarrow b$ and $b\rightarrow a \rightarrow b \rightarrow a$, i.e. as $\Pi\,Q / \langle J\rangle^3$, for $J$ the ideal generated by all arrows.  For quotients of path algebras by powers of the arrow ideal, the bar complex has large acyclic subcomplexes, which means that Hochschild cohomology can be computed rather tractably. 
Starting from here, Lekili and Perutz \cite{LekiliPerutz} determined all possible $A_{\infty}$-structures on the 6-dimensional algebra $A$.

\begin{Theorem}[Lekili-Perutz]\label{Thm:LPelliptic}
Suppose char$(\bK) \not\in \{2,3\}$. Any minimal $A_{\infty}$-structure on $A$ is gauge-equivalent to 
the subcategory of the derived category of sheaves on a unique cubic curve in the family
\[
\{y^2z = 4x^3 - pxz^2 - qz^3 \} \subset \mathbb{A}^2_{p,q}  \times \bP^2
\]
generated by the structure sheaf $\mathcal{O}$ and skyscraper sheaf $\mathcal{S}_{[0,1,0]}$.
\end{Theorem} 
Lekili and Perutz also show that if one takes a point on the discriminant $\{p^3 = 27 q^2\}$ where the cubic curve acquires a node, the category reproduces the subcategory of $\scrF(T^2 \backslash \{pt\})$ generated by $\{l,m\} \subset T^2\backslash \{pt\}$, which in particular shows that the latter category is not formal -- even though the punctured 2-torus is exact and the only non-trivial holomorphic polygons are constant polygons (which, crucially, are not transversely cut out, so must be counted ``virtually").  This illustrates the difficulty in making direct computations in the Fukaya category without resort to algebraic classification results.

\subsection{The genus two surface\label{Subsec:genus2}}  Consider a hyperelliptic curve $\Sigma_2 \rightarrow \bP^1$ branched at the origin and the fifth  roots of unity.  Fix a primitive $\theta$ for the pullback of $\omega_{\Sigma_2}$ to the unit circle bundle of $T\Sigma_2$.  Although not monotone, the curve $\Sigma_2$ has a well-defined ``balanced" Fukaya category (linear over $\bC$) whose objects are Lagrangian circles $\gamma: S^1 \hookrightarrow \Sigma_2$ for which $\int_{\gamma} (D\gamma)^*\theta = 0$.  (Non-contractible isotopy classes of curves  contain unique balanced representatives; energy and index are correlated for polygons with balanced boundary conditions; finally, the category $\scrF(\Sigma_2)$ is independent of the choice of $\theta$ up to quasi-isomorphism.)  A collection of $5$ balanced curves $\gamma_i$ is given by the matching paths depicted on the left of  Figure \ref{Fig:BalancedSphere}, and these split-generate the Fukaya category by another application of Proposition \ref{Prop:SpheresGenerate}.

\begin{center}
\begin{figure}[ht]
\includegraphics[scale=0.4]{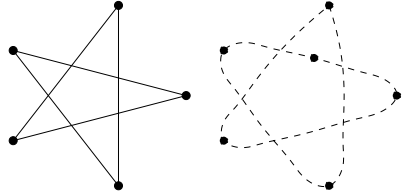}
\caption{Generators for $\scrF(\Sigma_2)$ and a holomorphic hexagon\label{Fig:BalancedSphere}}
\end{figure} 
\end{center}

The choice of the  $\gamma_i$ as split-generators is motivated by Seidel's observation \cite{Seidel:HMSgenus2}  that
\[
\oplus_{i,j} HF^*(\gamma_i, \gamma_j) \ \cong \ \Lambda^*(V) \rtimes \bZ_{5}
\]
for a 3-dimensional vector space $V$.  $A_{\infty}$-structures can therefore be approached using the tools of Section \ref{Subsec:FiniteDeterminacy} (which admit straightforward equivariant versions in the presence of a finite group action), i.e. described in terms of polyvector fields.  Direct computation shows  the $A_{\infty}$-structure is given by a pair $(W,\eta)$ comprising a formal function and formal 2-form with 
\[
W=z_1z_2z_3 + z_1^{5} + z_2^{5} + z_3^{5} + O(6)
\]
(we are ignoring signs, which can be massaged by suitable changes of basis in the Floer complexes). The low-order non-trivial coefficients arise from higher multiplications $\mu^3$ and $\mu^5$ determined by ``manifest" holomorphic polygons: the triangles which sit in the  corners of the pentagram, or the pentagons such as that whose boundary is sketched on the right side of Figure \ref{Fig:BalancedSphere} (projecting to $\bP^1$ with multiplicity two over the inner pentagon). In each case these polygons have an additional boundary marked point at which the map is smooth but constrained to lie on a cycle on a given Lagrangian, which occurs as two adjacent boundary conditions in the associated quadrilateral respectively hexagon.  Since $W$ has an isolated singularity at $0$, it is equivalent by formal change of variables to a polynomial, and in fact to its degree $\leq 5$ Taylor series. This gives a complete description of the split-closure of $\scrF(\Sigma_2)$ in terms of modules over an explicit $A_{\infty}$-algebra. 

\subsection{The 4-torus\label{Sec:4torus}}  Take $\bK = \Lambda$. A special case of Theorem \ref{Thm:LPelliptic} (which was known earlier) shows that $D^{\pi}\scrF(T^2) \simeq D^b(E)$, where $T^2$ is the square symplectic torus and $E$ the Tate curve.  Fukaya categories obey a kind of K\"unneth theorem  \cite{AbouzaidSmith}; using (the $A_{\infty}$-refinement of) Wehrheim-Woodward's  quilt theory, Section \ref{Sec:Quilt} and \eqref{Eqn:QuiltFunctor}, one can identify $D^{\pi}\scrF(T^2 \times T^2)$ with the category of endofunctors of $D^{\pi}\scrF(T^2) \simeq D^b(E)$. That implies that $\scrF(T^4)$ is split-generated by the four Lagrangian two-tori $l\times l, l\times m, m\times l, m\times m$, and
\begin{equation} \label{Eqn:HMST4}
D^{\pi} \scrF(T^4) \simeq D^b(E\times E).
\end{equation}
\eqref{Eqn:HMST4} has concrete consequences for the symplectic topology of the 4-torus.  
Suppose $\phi: T^4 \rightarrow T^4$ is a symplectomorphism which acts trivially on cohomology. It then has a well-defined flux, meaning that for any 1-cycle $\gamma \in H_1(T^4)$, one can pick a 2-chain $C_{\gamma}$ with $\partial C_{\gamma} = \gamma \cup \phi(\gamma)$ and consider  $\int_{C_{\gamma}} \omega \in \bR / \bZ$, where dividing out by the periods of the symplectic form removes the non-uniqueness in the choice of $C_{\gamma}$.  

\begin{Proposition} If $\phi$ has trivial flux, then any symplectomorphism in the Hamiltonian isotopy class of $\phi$ with non-degenerate fixed points has at least $16$ fixed points.
\end{Proposition}

\begin{proof}
The autoequivalences of the category $D^b(E\times E)$ are completely understood through work of Mukai, Orlov and Polishchuk, see \cite{Orlov:abelian, Pol:weil};  those that act trivially on cohomology are generated, modulo shift, by tensoring by line bundles and by translations of $E\times E$.  On the symplectic side, these correspond to the autoequivalences which either translate $T^4$, and have non-zero flux, or equip Lagrangians with non-trivial local coefficients, which act trivially at the level of spaces, i.e. on the points of $T^4$. Hence, $\phi$ has the same Floer cohomology as the identity autoequivalence of $\scrF(T^4)$ (or symplectomorphism of $T^4$), which is just $H^*(T^4)$.
\end{proof}

\subsection{Quadric hypersurfaces}\label{Sec:Quadrics}

Let $X=Q \subset \bP^{2k+1}$ denote an even-dimensional quadric hypersurface. Any quadric has at most one node, so $X$ contains a distinguished Lagrangian sphere $L$ up to isotopy, the unique such which is the vanishing cycle of an algebraic degeneration. For $k>0$ this has minimal Maslov number $2c_1(Q) > 2$ so lies in the zero-summand $\scrF(Q;0)\subset \scrF(Q)$. 

\begin{enumerate} 
\item  Let $\scrX \rightarrow \bP^1$ denote the total space of a Lefschetz pencil of quadric hypersurfaces: all the vanishing cycles are isotopic to the fixed Lagrangian $L$. It is straightforward to check that the only holomorphic sections are those coming from blowing up the base-locus of the pencil, hence evaluate in a single fibre to $Q \cap Q' \subset Q$ which is proportional to the first Chern class. Hence Proposition \ref{Prop:SpheresGenerate} applies, and $L$ split-generates $\scrF(Q;0)$. 

\item Alternatively, consider the map $QH^*(Q;0) \rightarrow HF^*(L,L)$.  Both spaces have rank 2, the map is unital so non-trivial in degree $0$, and since $[L] \neq 0 \in H_{2k}(Q)$  the map is non-trivial in degree $2k$ (both sides are mod $4k$ graded), so the second version of the generation criterion also applies.
\end{enumerate} 

The algebra $HF^*(L,L)$ is semi-simple.  Generation by $L$ implies that $D^{\pi}\scrF(Q;0)$ is essentially trivial (equivalent to modules over a sum of two fields, i.e. pairs of vector spaces), see \cite{Smith:HFQuadrics}.   Any simply-connected Lagrangian $K \subset Q$ defines such a module, whose endomorphism algebra is on the one hand $HF^*(K,K)$ and on the other is a direct sum of two matrix algebras.  

Since $N_{min}(K) = 4k$, the pearl model of Proposition \ref{Lem:Oh} shows that the Oh spectral sequence collapses for grading reasons, and $HF^*(K,K) \cong H^*(K)$ additively.  A fairly simple algebro-geometric computation \cite{Beauville} shows that the quantum cohomology $QH^*(Q)$ contains an invertible element $a$ of degree $2k$. Since $\mathcal{CO}^0$ is a unital ring homomorphism, product with $\mathcal{CO}^0(a)$ defines an automorphism of (the $\bZ/4k$-graded group)  $HF^*(K,K)$, which, considering the pearl model for multiplication mentioned in the proof of Proposition \ref{Lem:Oh}, is given by  classical cup-product with $a|_K$ and corrections of degree $\leq 2k-N_{min}(K) = -2k$.  It follows that $H^*(K)$ vanishes in degrees $0<\ast< 2k$, so $K$ is a homology sphere.  Since moreover $H^*(K)$ is a sum of two matrix algebras, it must be the sum of two one-dimensional such algebras, so $K$ defines the same module as $L$.   It follows that $K$ is  isomorphic in $\scrF(Q;0)$ to $L$, whence $K \cap L \neq \emptyset$, and no two simply-connected Lagrangians can be Hamiltonian disjoined from one another.    

Any symplectomorphism of $Q$ co-incides in $\Auteq(\scrF(Q;0)) \cong \bZ_2$ with the identity or  the Dehn twist $\tau_L$ which swaps the two idempotent summands of $HF^*(L,L)$. Since the coarse moduli space $\EuScript{Q}$ of quadrics has fundamental group $\bZ_2$, this means that the map of Question \ref{Question:Monodromy} is injective in this case.

\subsection{Toric Fano varieties} Let $X$ be a toric Fano variety, with moment map $\mu: X \rightarrow \frak{t}^*$ and moment image $\Delta$.  
Maslov 2 discs with boundary on a fibre of the moment map are enumerated by the codimension one facets of the moment polytope $\Delta$, the counts of which can be assembled into a Laurent polynomial $W: \mathrm{int}(\Delta) \rightarrow \Lambda$ whose monomial terms are determined by the normal vectors to the facets of $\Delta$, see \cite{Cho,Cho-Oh}.  (The notation reflects the fact that this fits into the general picture of potential functions mentioned in passing in Section \ref{Sec:Mirror}. The divisor $D$ is the toric boundary of $X$, the moment map defines the Lagrangian fibration of $X\backslash D$, and the mirror $\Cech{U} \cong (\bC^*)^n$.   For $\bP^2$, $W(x,y) = x+y+\lambda/xy$.) Example \ref{Ex:ToricFibre} implies that the Floer cohomology of the given fibre of the moment map is non-trivial with some local system, i.e. $\partial_1 = 0$,  if and only if the relevant point of the moment polytope is a critical value of $W$.   

Suppose $X$ has even complex dimension $n=2k$ and the spectrum of quantum product $\ast c_1(X)$ is simple, so $QH^*(X) = \oplus_{\lambda} QH^*(X;\lambda)$ splits into a sum of fields. 
There are then $\mathrm{rk}(QH^*(X))$ distinct local systems $\xi_j \rightarrow L$ over the toric fibre $L$ lying over the barycentre  $b\in\Delta$ with respect to which $HF^*(\xi_j, \xi_j)$ is non-zero, and in fact in each case then a Clifford algebra (see e.g. \cite{Cho-Oh, FO3:toric, RitterSmith}). The object $(\xi_j \rightarrow L)$ splits into a sum of idempotents which are pairwise isomorphic up to shift; the even-dimensional Clifford algebra is a matrix algebra, and the closed-open map
\[
QH^*(X;\lambda_j) \longrightarrow \bK \subset HH^*(Cl_n, Cl_n) 
\] is an isomorphism onto the endomorphism subalgebra of any given idempotent summand, since the map is unital and both sides have rank 1. Theorem \ref{Thm:Generation} implies that the Fukaya category is split-generated by $(L,\xi_j)$, so any  Lagrangian $K$ with $HF(K,K)\neq 0$ must intersect $\mu^{-1}(b)$.  


\subsection{The quartic surface}  \label{Sec:K3} Let $Q= \{z_0^4+\cdots + z_3^4 = 0\}$ be the quartic $K3$ surface.  The pencil of quartics defined by
\begin{equation}\label{pencil}
\lambda(z_0^4+ \cdots z_3^4) + z_0z_1z_2z_3
\end{equation}
is Lefschetz away from the fibre at $\lambda = 0$, which is a union of co-ordinate hyperplanes.  There are 64 nodal fibres, which yields a collection of 64 Lagrangian spheres $V_i$ in $Q$. A word in positive Dehn twists in these spheres acts on $\scrF(Q)$ by a shift \cite{Seidel:HMSquartic}, which implies by a variant of  Proposition \ref{Prop:SpheresGenerate} that these spheres split-generate (the map in \eqref{eqn:concatenate} now vanishes for degree reasons).  Exploiting the symmetry properties of the situation, Seidel proved that
\[
\scrA(Q) \, = \, \oplus_{i,j} HF^*(V_i, V_j) \, \cong \, \Lambda(W) \rtimes \Gamma_{64}
\]
for a four-dimensional complex vector space $W$ and $\Gamma_{64} \subset H \subset SL(W)$ the finite subgroup of order 4 elements of a maximal torus $H$ of $SL(W)$.  The Hochschild cohomology $HH^*(\Lambda(W)\rtimes \Gamma_{64})$ contains a summand $HH^*(\Lambda(W))^{\Gamma_{64}}$ of equivariant polyvector fields.  The subgroup corresponding to non-curved deformations (where the Lagrangians stay unobstructed) has rank 7 in degree $2$, and there is a \emph{unique} class invariant under the $\bZ_4$-subgroup of $GL(W)$ which cyclically permutes the co-ordinates. This symmetry preserves the pencil \eqref{pencil}, and the $A_{\infty}$-structure on $\scrF(Q)$ can be taken invariant in this sense.  That means that at any given order in the Novikov parameter, the relevant deformation is unique up to scale (spanned by the invariant polynomial $y_0y_1y_2y_3$, with $y_i$ co-ordinates on $W^*$).  The upshot is that there is a unique $\bZ_4$-equivariant  non-formal $A_{\infty}$-structure on $\scrA(Q)$ up to gauge equivalence.  Seidel checks non-formality \cite{Seidel:HMSquartic}, and that pins down the quasi-isomorphism class of $D^{\pi}\scrF(Q)$. 

In both this case and that of the genus 2 curve considered previously, the appearance of an exterior algebra is rather magical. There is a more conceptual explanation \cite{Seidel:HMSgenus2, Sheridan}, coming from a canonical immersed Lagrangian $P \subset \bP^n\backslash \{(n+2) \, \mathrm{hyperplanes}\}$ which has exterior algebra Floer cohomology, and for which the collections of circles $\{\gamma_i\} \subset \Sigma_2$ or spheres $\{V_i\} \subset Q$ arise as the total preimage of $P$ under a ramified covering of projective space.  The construction of $P$ is in turn inspired by ideas related to the SYZ viewpoint on mirror symmetry.

\begin{Notes} For $T^2$, see \cite{PolishchukZaslow, PZ2, LekiliPerutz, LekiliPerutz2}. For superpotentials and toric varieties, see \cite{Auroux:slag, Cho, Cho-Oh, FO3:toric, RitterSmith, Woodward}. For Fukaya categories of higher genus surfaces, Calabi-Yau or Fano hypersurfaces and some complete intersections, see \cite{Efimov, Sheridan2, Sheridan3, Smith:HFQuadrics}.
\end{Notes}

\section{Further directions}

 Once one accepts the Fukaya category as  an object of basic interest in symplectic topology, it is natural to consider invariants of symplectic manifolds defined through that category, or properties of symplectic manifolds which are natural from the viewpoint of the Fukaya category.

\renewcommand{\P}{\mathcal{P}}

\subsection{Stability conditions}  A triangulated category $\CC$ has an associated complex manifold Stab$(\CC)$ of stability conditions \cite{Bridgeland}, which under mild hypotheses on $\CC$ is finite-dimensional.  Stab($\CC)$ carries an action of the group $G(\CC)=$Auteq$(\CC)$ of autoequivalences, and can be viewed as an attempt to build directly from $\CC$ a geometric model for the classifying space $BG(\CC)$: in the few known cases Stab$(\CC)$ has or is conjectured to have contractible universal cover (cf. \cite{Woolf}). 

Let $K(\CC)$ denote the Grothendieck group of $\CC$, the free abelian group on  $\Ob \, \CC$ modulo relations from exact triangles $[\mathrm{Cone}(a: A \rightarrow B)] = [B]-[A]$.  A \emph{stability condition} $\sigma=(Z,\P)$ on   $\CC$
consists of
a group homomorphism
$Z\colon K(\CC)\to\bC$ called the \emph{central charge},
and full additive
subcategories $\P(\phi)\subset\CC$ for each $\phi\in\bR$,
which together satisfy the following axioms:
\begin{itemize}
\item[(a)] if $E\in \P(\phi)$ then $Z(E)\in \bR_{>0}\cdot e^{i\pi\phi}\subset \bC,$ \smallskip
\item[(b)] for all $\phi\in\bR$, $\P(\phi+1)=\P(\phi)[1]$,\smallskip
\item[(c)] if $\phi_1>\phi_2$ and $A_j\in\P(\phi_j)$ then $\Hom_{\CC}(A_1,A_2)=0$,\smallskip
\item[(d)] for each $0 \neq E\in\CC$ there is a finite sequence of real
numbers $\phi_1>\phi_2> \dots >\phi_k$ 
and a collection of triangles
\[
\xymatrix@C=.5em{
0_{\ } \ar@{=}[r] & E_0 \ar[rrrr] &&&& E_1 \ar[rrrr] \ar[dll] &&&& E_2
\ar[rr] \ar[dll] && \ldots \ar[rr] && E_{k-1}
\ar[rrrr] &&&& E_k \ar[dll] \ar@{=}[r] &  E_{\ } \\
&&& A_1 \ar@{-->}[ull] &&&& A_2 \ar@{-->}[ull] &&&&&&&& A_k \ar@{-->}[ull] 
}
\]
with $A_j\in\P(\phi_j)$ for all $j$;
\item[(e)] (the support property) for some norm $\|\cdot\|$ on $K(\D)\tensor\bR$ there is a constant $C>0$ such that $\|\gamma\|< C\cdot |Z(\gamma)|$
for all classes $\gamma\in K(\D)$  represented by $\sigma$-semistable objects in $\D$.
\end{itemize}

\begin{Theorem}[Bridgeland] 
\label{basic}
The space $\Stab(\CC)$ has the structure of a complex manifold, such that the forgetful map
\[\pi\colon \Stab(\CC)\lra \Hom_{\bZ}(K(\CC),\bC)\]
taking a stability condition to its central charge,
is a local isomorphism.
\end{Theorem}

Since Stab$(\CC)$ is locally modelled on a fixed vector space which has an integral lattice,  it inherits a great deal of geometry: a flat K\"ahler metric, an  affine structure, a Lebesgue measure, etc.  The group of triangulated autoequivalences $\Auteq(\CC)$ acts on $\Stab(\CC)$ preserving all this structure. It is important that the category $\CC$ is $\bZ$-graded, i.e. that the shift functor is not periodic, so this theory is relevant to symplectic manifolds with $2c_1=0$.

\begin{Example} 
Let $Y \rightarrow \bC$ be the affine $A_n$-threefold, which is a Lefschetz fibration with generic fibre $T^*S^2$,  and with $n+1 \geq 3$ singular fibres.  $Y$ then retracts to an $A_n$-chain of Lagrangian spheres.  There is a connected component of $\mathrm{Stab}(\scrF(Y))$ whose quotient by the subgroup of autoequivalences preserving that component is isomorphic to $\mathrm{Conf}_0^{n+1}(\bC) / \bZ_{n+3}$ (where the cyclic group acts by multiplication by an $(n+3)$-rd root of unity).  See \cite{BridgelandSmith}. \end{Example}

Let $\{\sum_{i=0}^3 z_i^4 = 0\} = Q\subset \bP^3$ be the Fermat quartic $K3$ surface. This carries an action of the group $\bZ_4^4$ by multiplying co-ordinates by roots of unity; denote by $\Gamma \cong \bZ_4^3$ the subgroup of elements which preserve the holomorphic volume form (so $x_i \mapsto \xi_i x_i$ with $\prod_i \xi_i = 1$). Let $\Cech{Q}$ be a crepant resolution of the orbifold $Q / \Gamma$. A conjectural case of mirror symmetry ``dual" to \cite{Seidel:HMSquartic} asserts that (for a suitable K\"ahler form on $\Cech{Q}$)  
\begin{equation} \label{HMS?}
D^{b}(Q) \stackrel{?}{\simeq} D^{\pi}\scrF(\Cech{Q}).
\end{equation}  Indeed by varying the K\"ahler parameters of $\Cech{Q}$ one expects to obtain Fukaya categories dual to derived categories of sheaves on quartics $Q'$ arising as small complex deformations of $Q$. There are representations
\[
\pi_1(\mathcal{M}(\Cech{Q})) \longrightarrow \pi_0\Symp(\Cech{Q}) \longrightarrow \Auteq(D^{\pi}\scrF(\Cech{Q})).
\]
Bridgeland and Bayer \cite{Bridgeland:K3, BayerBridgeland} computed Stab$(Q')$, where $Q'$ is a generic (Picard rank 1) quartic,  and thereby $\Auteq(D^b(Q'))$ (e.g. the ``Torelli group" of autoequivalences acting trivially on $K$-theory, modulo shifts, is an infinite rank free group).  A proof of \eqref{HMS?} would imply that, for suitably generic K\"ahler forms, all ($\bZ$-graded, Calabi-Yau) autoequivalences of $D^{\pi} \scrF(\Cech{Q})$ arise from monodromy of algebraic families (contrast to Section \ref{Sec:MCGIII} and \cite{Smith:HFQuadrics}).

\subsection{Orlov spectrum}  A \emph{strong generator} for a split-closed triangulated category $\CC$ is an object $E \in \CC$ for which every object of $\CC$ is a summand of some twisted complex on $E$, i.e. for which $\Tw^{\pi}\langle E\rangle = \CC$, where $\langle \cdot \rangle$ denotes the triangulated subcategory generated by repeatedly taking shifts and cones. Given a strong generator, its generation time $U(E)$ is the least $n$ such that the entire category $\CC$ is obtained by taking summands of twisted complexes of length at most $n+1$ (i.e. involving taking $n$ cones) The \emph{Orlov spectrum} of a triangulated category is 
\[
OSpec(\CC) \, = \, \left\{U(E) \ | \ E \in \CC, U(E) < \infty \right\} \ \subset \bZ_{\geq 0}.
\]
For instance, if $\CC = D^b(Z)$ is the derived category of coherent sheaves, it essentially follows from the existence of resolutions of sheaves by vector bundles which, in turn, are globally generated after sufficient twisting, that a finite sum of powers of an ample line bundle is a strong generator. The infimum and supremum of the Orlov spectrum give two different notions of the \emph{dimension} of a triangulated category.

\begin{Example} The ``shift" version of Proposition \ref{Prop:SpheresGenerate} mentioned in Section \ref{Sec:K3} implies \cite{BallardEtAl} that if a collection of spherical objects $S_1,\ldots, S_n$ in $\CC$ satisfy a positive relation $T_{S_{i_1}}\cdots T_{s_{i_r}} = [k]$ for some $k$, then $\oplus_i S_i$ is a strong generator with generation time at most $r-1$.  Such positive relations can be constructed from classical Dehn twists on curves, and it follows that $\scrF(\Sigma_g)$ always admits a strong generator given by the direct sum $G$ of $2g$ simple closed curves which form an $A_{2g}$-chain and with generation time $4g \leq U(G) \leq 8g$. 
\end{Example}

\begin{Example}
Let $\Sigma_2$ be a curve of genus 2. From Section \ref{Subsec:genus2}, the category $\scrF(\Sigma_2)$ is equivalent to the $A_{\infty}$-structure  on $\Lambda^*(\bC^3)$ induced by the Maurer-Cartan element $z_1z_2z_3 + z_1^5+z_2^5+z_3^5$.  Constraints on the Orlov spectra of hypersurface singularities were obtained  in \cite{BallardEtAl}, and the supremum of the Orlov spectrum in this case is $< 29$ (more generally $<12g+5$ for $\Sigma_g$).
\end{Example}

Sophisticated potential relations between Orlov spectra and classical algebro-geometric rationality questions have been extensively explored and promulgated by Katzarkov, see  \cite{Katzarkov,BallardEtAl} (such relations essentially stem from constraints upon the Orlov spectrum imposed by semi-orthogonal decompositions of categories arising from blow-ups).  The geometric significance of the Orlov spectrum of a Fukaya category remains largely mysterious, although explicit lower bounds would give constraints on possible  positive relations amongst Dehn twists, and hence on the topology of Lefschetz fibrations with a given symplectic manifold as fibre.  

\subsection{Non-commutative vector fields}  For any $a \in H^1(X;\bR)$ there is a Floer cohomology group associated to the end-point $\phi_a$ of a path of symplectomorphisms with flux $a$.  There is thus a basic connection between $H^1$ and symplectic dynamics. On the other hand, in Gromov-Witten theory or viewed as the domain of $\mathcal{CO}$, $QH^*(X)$ is typically only mod-2 graded, and all the odd degree cohomology of $X$ lives on (more-or-less) the same footing. With this motivation, Seidel introduced new flux-type invariants which probe the symplectic dynamics of $H^{odd}(X)$. 

The philosophical big picture might run as follows.  An $A_{\infty}$-category $\CC$ has a ``derived Picard group", with tangent space $HH^1(\CC,\CC)$.  Given a particular $HH^1$-class (``non-commutative vector field"), one can consider whether or not it integrates to a ``periodic flow" on the  ``moduli space of objects" in $\CC$.  Making this precise has obvious pitfalls:  even in a model situation (for instance $\CC$ the derived category of sheaves on a smooth variety $Z$ and the $HH^1$-class coming from $H^0(TZ)$), one does not expect  a polynomial vector field to integrate to an algebraic flow.  Seidel's insight is that one can nonetheless make sense of algebraic orbits of the (non-existent) flow.  

\newcommand{\scrP}{\EuScript{P}}

It is technically easier to define a notion of ``doubly periodic" orbits. Fix an elliptic curve $\bar{S}$ over the Novikov field, together with a nowhere zero differential $\eta \in \Omega^1_{\bar{S}}$.  Working Zariski-locally over $\bar{S}$,  the notion of a twisted complex of objects of $\CC$ generalises to a complex whose coefficients are modules over the ring of functions $\mathcal{R}$ of an affine $S \subset \bar{S}$, i.e.  vector bundles over $S$.   Families of objects of $\CC$ over $S$ are functors from $\CC^{opp}$ to the $dg$-category of complexes of projective $\mathcal{R}$-modules.  This leads to an extended category $\CC_{\mathcal{R}}$, which has an associated derived category Perf$(\CC_{\mathcal{R}})$ of perfect modules. 
Projective $\mathcal{R}$-modules $\scrG$ admit essentially unique connections, i.e. maps $\scrG \rightarrow \Omega^1_{S} \otimes \scrG$ satisfying the Leibniz property. A family with connection $\scrP$  has an associated infintesimal deformation class Def$(\scrP)$, given by using the connection to differentiate the internal differential in the underlying twisted complex.  This leads to a natural map
\begin{equation} \label{Eqn:DifferentiateConnection}
\mathrm{Def}: HH^{odd}(\CC_{\mathcal{R}}, \CC_{\mathcal{R}}) \longrightarrow H^1(hom_{\mathrm{Perf}(\CC_{\mathcal{R}})}(\scrP, \scrP))
\end{equation}
for any family $\scrP \in \CC_{\mathcal{R}}$.  On the other hand, there is a natural map of algebraic vector bundles
\begin{equation} \label{Eqn:FamilyDeviation}
\nabla: TS \longrightarrow H^1(hom_{\mathrm{Perf}(\CC_{\mathcal{R}})}(\scrP, \scrP))
\end{equation}
which measures how the family of objects deforms in a given direction in $S$.  We say that an object $\scrP$ \emph{follows the deformation field} $[a] \in HH^{odd}(\CC_{\mathcal{R}}, \CC_{\mathcal{R}})$ if \eqref{Eqn:DifferentiateConnection} and \eqref{Eqn:FamilyDeviation} are related by
\[
\nabla \ = \ \eta \otimes \mathrm{Def}([a])
\]
where $\eta$ is the fixed 1-form on $\bar{S}$.  A family which follows the  deformation field is regarded as belonging to an algebraic orbit of the formal flow associated to $[a]$.  We then define the subset
\[
\mathrm{Per}(\CC) \subset HH^{odd}(\CC,\CC)
\]
of periodic elements to be those classes $[a]$ for which every object $A$ extends to a family $\scrA$ over an affine $S\subset \bar{S}$  which follows the field $[a]$. Thus, periodic classes are ones for which \emph{all} objects can be made the special fibres of families over $S\subset \bar{S}$.  

For this to be useful, one needs ways of showing classes are or are not periodic. 
If $X \rightarrow \bar{S}$ is a trivial symplectic fibre bundle, the algebraic notion of family of objects is sufficiently close to geometry to infer that trivial product  families of Lagrangian submanifolds are indeed families of objects, and to conclude that classes coming from $H^1(\bar{S})$ are periodic.  On the other hand, if $X\rightarrow \bar{S}$ is a non-trivial mapping torus, one can hope to show that classes in $H^1(\bar{S})$ are \emph{not} periodic by Floer cohomology computations.  Going back to the question of existence versus uniqueness of flows for algebraic vector fields, Seidel shows that if a family $\scrA$ following a deformation field $[a]$ exists, it is unique, and hence its pullback to any finite cover is unique.  Using the existence of symplectomorphisms $f$ (arising as monodromy of $X\rightarrow \bar{S}$) for which $HF^*(f^m)$ carries a non-trivial $\bZ_m$-action, one can aim to  produce distinct families following a fixed deformation on an $m$-fold covering of $X$, violating periodicity of a particular class on $X$ itself.   Whilst these examples amount to embedding classical  flux in the more abstract notion of periodicity, the pay-off is that periodic elements are an invariant of the Fukaya category, which gives some robustness under surgery operations which kill the fundamental group (subcritical handle addition in the exact case, or passing to certain blow-ups in the closed case).   

\begin{Example} Let $K \subset \bP^3$ be a quartic $K3$ surface. Let $X_1 = K^2 \times T^2$ and $X_2 = S^1 \times Y$, with $Y$ the mapping torus of a symplectomorphism $f$ of $K^2$ which is differentiably but not symplectically trivial (this is a more complicated version of the squared Dehn twist).  One can embed the $X_i$ into $K^7$ by an $h$-principle. Seidel \cite{Seidel:Flux} shows that  for suitable $f$ the blow-ups $Z_1 = Bl_{X_1}(K^7)$ and $Z_2 = Bl_{X_2}(K^7)$ are symplectically distinct, although they are diffeomorphic and deformation equivalent (hence have the same Gromov-Witten invariants, quantum cohomology etc).  They are distinguished by the subsets $\mathrm{Per}(Z_i) \subset QH^{odd}(Z_i) \cong HH^{odd}(\scrF(Z_i),\scrF(Z_i))$; for $Z_1$ both classes in $H^3(Z_1)$ coming from $H^1(T^2)$ are periodic, whilst for $Z_2$ the periodic subset intersects that 2-dimensional lattice in a single copy of $\bZ$. 
\end{Example}

\begin{Notes}
For stability conditions on Fukaya categories, see \cite{KS, Smith:quiver, Thomas:braid}. For Katzarkov's program, see \cite{BallardEtAl,KatzarkovKerr, Katzarkov}. For non-commutative vector fields, see \cite{Seidel-Solomon, Seidel:Flux, AbouzaidSmith:arc}.
\end{Notes}


\bibliographystyle{amsplain}

\end{document}